\documentclass[reqno]{amsart}
\usepackage{amssymb,latexsym,color,}
\usepackage{amsmath}
\usepackage{amsthm}
\usepackage{graphicx}
\usepackage{palatino,mathpazo}
\usepackage{titletoc}
\numberwithin{equation}{section}
\newtheorem{theorem}{Theorem}[section]

\newtheorem{lemma}[theorem]{Lemma}
\newtheorem{remark}[theorem]{Remark}

\newtheorem{definition}[theorem]{Definition}

\theoremstyle{definition}

\def\t{\widetilde}

\def\XXint#1#2#3{{\setbox0=\hbox{$#1{#2#3}{\int}$}
     \vcenter{\hbox{$#2#3$}}\kern-.5\wd0}}

\newcommand{\RN}{\mathbb{R}^2}
\newcommand{\intr}{\int_{\mathbb{R}^2}}

\newcommand{\e}{\varepsilon}

\begin{document}
\title[bubbling solutions without mass concentration]
{Existence of bubbling solutions without mass concentration}

\author{Youngae Lee}
\address{Youngae ~Lee,~Department of Mathematics Education, Teachers College, Kyungpook National University, Daegu, South Korea}
\email{youngaelee@knu.ac.kr}
\author{ Chang-shou Lin}
\address{ Chang-shou ~Lin,~Taida Institute for Mathematical Sciences, Center for Advanced Study in
Theoretical Sciences, National Taiwan University, Taipei 106, Taiwan}
\email{cslin@math.ntu.edu.tw}
\author{Wen Yang}
\address{ Wen ~Yang,~Wuhan Institute of Physics and Mathematics, Chinese Academy of Sciences, P.O. Box 71010, Wuhan 430071, P. R. China}
\email{wyang@wipm.ac.cn}

\begin{abstract}The seminal work \cite{bm} by  Brezis and Merle has been  pioneering in   studying the bubbling phenomena of the mean field equation with singular sources. When the vortex points are not collapsing, the mean field equation possesses the property of the so-called "bubbling implies mass concentration".   Recently, Lin and Tarantello  in \cite{lt} pointed out that  the  "bubbling implies mass concentration" phenomena might not hold in general if the  collapse of singularities occurs.
In this paper, we shall construct the first concrete example of non-concentrated bubbling solution of the mean field equation with collapsing singularities.
\end{abstract}

 \maketitle

\section{Introduction}Let $(M,g)$ be a compact Riemann surface with volume $1$, and $\rho>0$ be a real number. We consider the following mean field type equation:
\begin{equation}
\label{mfd}
\Delta u+\rho\left(\frac{h_*e^{u}}{\int_Mh_*e^{u}\mathrm{d}v_g}-1\right)=4\pi\sum_{q_i\in S}\alpha_{i}(\delta_{q_i}-1)\ \textrm{on} \ M,
\end{equation}
where $\Delta$ is the corresponding Laplace-Beltrami operator,  $S\subseteq M$ is a finite set of distinct points $q_i$, $\alpha_{q_i}>-1$, and $\delta_{q_i}$ is the Dirac measure at $q_i\in S$. The point $q_i\in S$ is called vortex point or singular source. Throughout this paper, we always assume $h_*>0$ and $h_*\in C^{2,\sigma}(M)$.

The equation \eqref{mfd} arises in various areas of mathematics and physics.
In conformal geometry, the equation \eqref{mfd} is related to the Nirenberg problem of finding  prescribed Gaussian curvature
if $S=\emptyset$, and the existence of a positive constant curvature metric with conic
singularities if $S\neq\emptyset$ (see  \cite{cy,t} and the references therein). Moreover, if the parameter $\rho=4\pi\sum\alpha_i$ and $M$ is a flat surface (for example, a flat torus), the equation \eqref{mfd} is an integrable system, which  is related to the classical Lame equation and the Painleve VI equation (see  \cite{clw, cklw}  for the details). The equation \eqref{mfd} is also related to the self-dual equation of the relativistic Chern-Simons-Higgs model.
For the recent developments related to \eqref{mfd}, we refer to the readers to   \cite{bamar, bmar, bt1, bclt, bm,  ck,cl1, cl2, cl3, cl4,  l1,   ly1,  m1, m2, m3, m4, nt1, nt2,t, y,y2} and references therein.

The seminal paper  \cite{bm} by  Brezis and Merle had initiated to study the blow
up behavior of solutions of \eqref{mfd}. Among others, they showed the
following "bubbling implies mass concentration" result:

\vspace{0.3cm}
\noindent {\bf{Theorem A}}. \cite{bm}  {\em  Suppose $S=\emptyset$ and $u_k$ is a sequence of blow up solutions to \eqref{mfd} with $\rho_k$. Then  there is a non-empty finite set $\mathcal{B}$ (namely, blow up set)
such that}
 \[\rho_k\frac{h_*e^{u_k}}{\int_Mh_*e^{u_k}\mathrm{d}v_g}\to\sum_{p\in \mathcal{B}}\beta_p\delta_p\ \ \textrm{as}\ \ k\to+\infty,\ \ \textit{where}\ \ \beta_p\ge 4\pi.\]
\vspace{0.3cm}

It was conjectured in \cite{bm} that if $S=\emptyset$, the  local mass $\beta_p$ at each blow up point $p\in \mathcal{B}$ satisfies $\beta_p\in 8\pi\mathbb{N}$, where $\mathbb{N}$ is the set of natural numbers.
This conjecture has been successfully
proved by Li and Shafrir  in \cite{ls}, and  Li in \cite{l}  showed that the local mass $\beta_p$ equals $8\pi$ exactly. For the equation \eqref{mfd}, $\rho\frac{h_*e^{u}}{\int_Mh_*e^{u}\mathrm{d}v_g}$ is called the \textit{mass distribution} of the solution $u$. Thus Theorem A  just says that if ${u}_k$ blows up as $k\to+\infty$, then the  mass is \textit{concentrated}. Later, Bartolucci and Tarantello in \cite{bt1} have extended Theorem A to include the case $S\neq\emptyset$. They also  proved that if a blow up point coincides with some singular point $q_i\in S$, then $\beta_{q_i}=8\pi(1+\alpha_{q_i})$ (see also \cite{bclt}).

Recently, Lin and Tarantello in \cite{lt} found a new phenomena such that if some of the vortices in \eqref{mfd} are collapsing, then a sequence of  blow up solutions  might not concentrate its mass. Indeed, they considered the following  equation:
\begin{equation}
\label{mfd_t}
\Delta \overline{u}_t+\rho\left(\frac{h_*e^{\overline{u}_t}}{\int_Mh_*e^{\overline{u}_t}\mathrm{d}v_g}-1\right)=4\pi\sum_{i=1}^2\alpha_i(\delta_{q_i(t)}-1)+ 4\pi\sum_{i=3}^N\alpha_{i}(\delta_{q_i}-1)\ \textrm{in} \ M,
\end{equation}
where  $\lim_{t\to0}q_i(t)= \mathfrak{q}\notin\{q_3, \cdots, q_N\}$, $i=1, 2$,  and $q_1(t)\neq q_2(t)$. Then  the following theorem was stated in \cite{lt}:

\vspace{0.3cm}
\noindent {\bf{Theorem B}}. \cite{lt}  {\em  Assume   $\alpha_{i}\in\mathbb{N}$  and  $\rho\in(8\pi,16\pi)$.
Suppose that  $\overline{u}_t$ is a sequence of blow up solutions of \eqref{mfd_t} as $t\to0$. Then  $\overline{u}_t\to \overline{w}$ uniformly locally in $C^2(M\setminus\{\mathfrak{q}\})$, where $\overline{w}$ satisfies   }
\begin{equation*}
\begin{aligned}
\Delta \overline{w}+(\rho-8\pi)\Big(\frac{h_*e^{\overline{w}}}{\int_Mh_*e^{\overline{w}}}-1\Big) =4\pi(\alpha_{1}+\alpha_{2}-2)(\delta_{\mathfrak{q}}-1)+4\pi\sum_{i=3}^N\alpha_{i}(\delta_{q_i}-1)\ \ \textrm{in}\ \ M.
\end{aligned}
\end{equation*}

The proof of Theorem B was sketched in \cite{lt}. For the complete proof, see \cite{llty}. In Theorem B, there is no restriction on $\alpha_1$ and $\alpha_2$, because $\rho\in(8\pi,16\pi).$ If $\rho>16\pi$, then we have to put some conditions on $\alpha_1$ and $\alpha_2$ in order to extend Theorem B. Indeed, in \cite{llty}, we generalize   Theorem B and obtain a sharp estimate of $\overline{u}_t$ under some nondegenerate conditions. To state the result in \cite{llty}, we let $G(x,p)$ be the Green's function of $-\Delta$ on $M$ satisfying
\begin{equation*}
-\Delta G(x,p)=\delta_p-1,~\int_MG=0.
\end{equation*} Throughout this paper, we  fix a point $\mathfrak{q}\in M$, and without loss of generality, we may choose a suitable coordinate centered at $\mathfrak{q}$ and \[\mathfrak{q}=0,\ \ q_1(t)=t\vec{e},\ \ q_2(t)=-t\vec{e},\ \ \textrm{ where}\ \ \vec{e}\ \ \textrm{ is a fixed unit vector in}\ \ \mathbb{S}^1.\]
To simplify our argument, we assume that
\begin{equation}
\label{simplify-assumption}
\alpha_1=\alpha_2=1.
\end{equation}
Now we consider the following equation, which is equivalent to \eqref{mfd_t}:
\begin{equation}
\label{2.1}
\left\{\begin{array}{l}
\Delta u_t+\rho\left(\frac{he^{u_t-G_t^{(2)}(x)}}{\int_Mh e^{u_t -G_t^{(2)}(x)}\mathrm{d}v_g }-1\right)=0,\\
\int_Mu_t\mathrm{d}v_g=0,\quad u_t\in C^{\infty}(M),
\end{array}\right.
\end{equation}
where
\begin{align}
\label{def_of_Gt}
&G_t^{(2)}(x)=4\pi   G(x,t\vec{e})+4\pi  G(x,-t\vec{e}),\ \ \mathrm{and}\\
h(x)=h_*(x)&\exp(-4\pi\sum_{i=3}^N\alpha_iG(x,q_i))\ge 0,\ \ h\in C^{2,\sigma}(M), \ \ h(0)>0.\label{prop_of_h}
\end{align}
We note $h(x)>0$ except at a finite set $S_0=\{q_3,\cdots,q_N\}$, where $0\not\in S_0$.
Now we can state the following result:
\medskip

\noindent {\bf Theorem C.} \cite{llty} {\em Assume $\alpha_i\in\mathbb{N},i\geq 3$, and  $\rho\notin 8\pi\mathbb{N}$. Suppose that  $u_t$ is a sequence of blow up solutions of (\ref{2.1}) as $t\to0$. Then $u_t\to w+8\pi  G(x,0)$ uniformly locally in $C^2(M\setminus\{0\})$, where $w$ satisfies
\begin{equation}
\label{thc.equationw}
\Delta w+(\rho-8\pi  )\left(\frac{he^{w}}{\int_Mhe^{w}}-1\right)=0\ \ \textrm{in}\ \ M, \ \ w\in C^2(M).
\end{equation} Furthermore, if the linearized equation of (\ref{thc.equationw}) at $w$ is non-degenerate, then for any $\tau\in(0,1)$, there is a constant $c_{\tau}>0$, independent of $t>0$, satisfying
\begin{align}
\label{thc.estimate}
\|u_t(x) -w(x)-8\pi  G(x,tp_t)\|_{C^1(M\setminus B_{2tR_0}(0))}\leq c_{\tau}t^{\tau},
\end{align}
where $tp_t$ is the maximum point of $u_t-w$ in $M$, and  $R_0>2$ is a fixed constant.}
\medskip

In Theorem C, the non-degenerate condition of $w$ is defined as follows:
\begin{definition}
A solution $w$ of \eqref{thc.equationw} is  non-degenerate if
$0$ is the unique solution of the following linearized problem:
\begin{equation}
\label{linear_pr_outside}
\left\{\begin{array}{l}
\Delta \phi+(\rho-8\pi)\frac{he^{w}}{\int_Mhe^{w}\mathrm{d}v_g}\left(\phi-\frac{\int_M he^{w}\phi \mathrm{d}v_g}{\int_Mhe^{w} \mathrm{d}v_g}\right)=0\ \ \mbox{in}\ \ M, \\
\int_M \phi \mathrm{d}v_g=0.
\end{array}\right.
\end{equation}
\end{definition}
 By the transversality theorem, we can always choose $h$ such that any solution   of \eqref{thc.equationw} is non-degenerate, i.e., the linearized equation \eqref{linear_pr_outside} admits only the trivial solution. We refer the readers to \cite[Theorem 4.1]{llwy} for the details of the proof.

We remark the estimate (\ref{thc.estimate}) which holds outside of a very tiny ball is rare in literature. Indeed, the non-degenerate assumption plays an important role in the proof of (\ref{thc.estimate}). In section 2, we shall review some estimates related to Theorem C.

In this article, inspired by Theorem B and Theorem C, we are interested in constructing a family of non-concentrated blow up solutions with the collapse of singular sources.  Our construction relies heavily on the non-degenerate assumption for (\ref{thc.equationw}).
Under the non-degenerate assumption for \eqref{thc.equationw} and the estimation established in  Theorem C, we could construct an accurate approximation solution and then succeed in obtaining the first concrete example of non-concentrated bubbling solution of \eqref{2.1} with collapsing singularities.
\begin{theorem}
\label{theorem1.4}
Let $h$ satisfies (\ref{prop_of_h}) with $\alpha_i\in\mathbb{N}$ for $i\ge 3$, and  $\rho\notin8\pi\mathbb{N}$.   Assume that  $w$ is a non-degenerate solution of \eqref{thc.equationw}. Then there is a small number $\emph{t}_0>0$ such that if $t\in(0,\emph{t}_0)$, then there is a  solution $u_t$ of \eqref{2.1} such that $u_t(x)$ blows up at $x=0$, and $u_t(x)$ converges to $w(x)+8\pi G(x,0)$ in $C^2_{\textrm{loc}}(M\setminus\{0\})$.
\end{theorem}

The construction of the bubbling solution without mass concentration is completely different from the previous ones  \cite{cl2, egp, emp, f, ly1}. Indeed,  the equation \eqref{2.1} can not be reduced to a singular perturbation problem:
\[\Delta u+ \varepsilon h e^{u}=0,\quad \ 0<\e\ll1,\]
which was treated by \cite{egp, emp, f} and \cite{ly1}, because the total mass of (\ref{2.1}) remains bounded, i.e.,
\[\lim_{t\to0}\int_Mhe^{u_t-G_t^{(2)}}\mathrm{d}v_g\neq+\infty.\]
Our construction is inspired by the ideas in \cite{cl2}. However, it is more complicate than the one treated in \cite{cl2} due to non-concentration of the mass distribution of $u_t$.

We remark that a related phenomena was also studied  by  D'Aprile,   Pistoia,  and Ruiz in \cite{dpr}, where the authors  proved  the existence of   solution for $2\times 2$ Toda system such that both components blow up at the same point, and only  one component has mass concentration, but the other one does not. However, their construction requires certain symmetry condition. We believe that our method in this paper might be able to construct such kind of solutions without symmetry condition. We will discuss later in a forthcoming project.

To prove Theorem \ref{theorem1.4}, at first, we find a suitable approximate solution by using the estimations in section \ref{sec_sharp}. After that, we have to prove the invertibility of the linearized operator $\mathbb{Q}_{t,q}\mathbb{L}_{t,q}$, which is the one of most important parts in this paper (see section \ref{sec_linear} for the definition of $\mathbb{Q}_{t,q}\mathbb{L}_{t,q}$). The most crucial step is to find the orthogonality condition for the linearized operator  $\mathbb{L}_{t,q}$ (see  $E_{\alpha,t,q,p}$ and  $F_{\alpha,t,q,p}$ in    Definition \ref{D3}). To do it,  we note that the blow up phenomena in Theorem C does require a double scaling (see \cite{llty} or section \ref{sec_sharp} below). In a very tiny ball $B_{tR_0}(tp_t)$, the second time re-scaled one from the original solution of \eqref{2.1} becomes a perturbation of an entire solution $v$ of Liouville equation:
\begin{equation}
\label{liouville}
\Delta v+ e^{v}=0\quad\textrm{in}\ \mathbb{R}^2.
\end{equation}
It is well known that  solutions to  \eqref{liouville} have been
completely classified by Prajapat and Tarantello in \cite{PT} such that%
\begin{equation}
v\left( z\right)=v_{\mu,a}(z) = \ln \frac{8e^{\mu}}{ ( 1+e^{\mu}  \vert z +a  \vert ^{2} ) ^{2}},  \label{17}
\end{equation}
where $\mu  \in \mathbb{R}$,  $  a =(a_1,a_2) \in \mathbb{R}^2$ can be arbitrary. So the equation \eqref{liouville} has invariance under dilations and translations.
The linearized operator $L$ for $v_{0,0}$  is defined by
\begin{equation}\label{entirelinear}L\phi:=\Delta \phi+\frac{8}{(1+|z|^2)^2}\phi\quad\textrm{in}\ \mathbb{R}^2.
\end{equation}
In \cite[Proposition 1]{bp}, it has been known that there are three kernels $Y_0$, $Y_1$, $Y_2$ for the linearized  operator $L$, where \begin{equation}
\begin{aligned}\left\{ \begin{array}{ll}\label{z01}Y_0(z):= \frac{1-|z|^2}{1+ |z|^2}=-1+\frac{2}{1+|z|^2}=\frac{\partial v_{\mu,a}}{\partial \mu}\Big|_{(\mu,a)=(0,0)},\\
Y_1(z):= \frac{z_1}{1+ |z|^2}=-\frac{1}{4}\frac{\partial v_{\mu,a}}{\partial a_1}\Big|_{(\mu,a)=(0,0)},
\\ Y_2(z):=\frac{z_2}{1+ |z|^2}=-\frac{1}{4}\frac{\partial v_{\mu,a}}{\partial a_2}\Big|_{(\mu,a)=(0,0)}.
\end{array}\right. \end{aligned}
\end{equation}
However, after a long computation, we found that due to the non-concentration of mass,  the orthogonality with $Y_0(z)$
should be not included in the finite-dimensional reduction like \cite{egp,f}. However, it causes a lot of difficulties in proving the invertibility of  the linearized equation. To overcome those difficulties, some of the  idea comes from our previous work on $SU(3)$ Chern-Simons system \cite{kll}. We consider this part as one of the main technical novelties in our paper.

Before ending the introduction, we would like to make some comments on the phenomena of the collapsing singularities. It arises naturally from
the study of the following Toda system:
\begin{equation}
\label{toda}
\left\{\begin{array}{l}
\Delta u_1+K_{11}\rho_1\Big(\frac{h_1e^{u_1}}{\int_Mh_1e^{u_1}dv_g}-1\Big)+K_{12}\rho_2\Big(\frac{h_2e^{u_2}}{\int_{M}h_2e^{u_2}dv_g}-1\Big)=0\\
\Delta u_2+K_{21}\rho_1\Big(\frac{h_1e^{u_1}}{\int_Mh_1e^{u_1}dv_g}-1\Big)+K_{22}\rho_2\Big(\frac{h_2e^{u_2}}{\int_{M}h_2e^{u_2}dv_g}-1\Big)=0
\end{array}\right.\ \mbox{in}\ M,
\end{equation}
where ${\bf{K}}=\left(\begin{array}{ll}K_{11}&K_{12}\\K_{21}&K_{22}\end{array}\right)$ is the Cartan matrix of rank two of the Lie algebra $\mathfrak{g},$ $\rho_i>0$, $h_i(x)=h_i^*(x)e^{-4\pi\sum_{p\in S_i}\alpha_{p,i}G(x,p)},$ $h_i^*>0$ in $M$,  $\alpha_{p,i}\in\mathbb{N}$, and $S_i$ is a finite set of distinct points in $M$, $i=1, 2$. 

In order to compute the topological degree for the  Toda system \eqref{toda}, we should calculate the degree jump due to blow up solutions  of Toda system. For example, let  $(u_{1k},u_{2k})$ be  a sequence of solutions of \eqref{toda} with $(\rho_{1k},\rho_{2k})\rightarrow(4\pi,\rho_2)$ satisfying $\rho_2\notin4\pi\mathbb{N}$ and $\max_{M}(u_{1k},u_{2k})\to+\infty$. In \cite[Theorem 1.7]{llyz}, it was proved that
\begin{align*}
\left\{\begin{array}{l}\rho_{1k}\frac{h_1e^{u_{1k}}}{\int_Mh_1e^{u_{1k}}}\rightarrow 4\pi\delta_Q,\   Q\in M\setminus S_1,\  \textrm{and}\\ u_{2k}\rightarrow {w}+4\pi K_{21} G(x,Q)\ \textrm{ in}\ C^{2,\alpha}_{loc}(M\setminus\{Q\}),\end{array}\right.
\end{align*}
where $({w},Q)$ is a solution of the so-called  shadow system of the Toda system:
\begin{align}
\label{shadowsystem}
\left\{\begin{array}{l}
\Delta {w}+2\rho_2\left(\frac{h_2e^{{w}+4\pi K_{21} G(x,Q)}}{\int_Mh_2e^{{w}+4\pi K_{21} G(x,Q)}}-1\right)=0,\\
\nabla\big(\log h_1e^{\frac{K_{12}}{2}{w}}\big)\mid_{x=Q}=0,~\mathrm{and}~Q\notin S_1,
\end{array}\right.
\end{align}
In \cite[Theorem 1.4]{llwy}, it was shown that  the  calculation of   the degree  contributed by  blow up solutions $(u_{1k},u_{2k})$ of  \eqref{toda} can be reduced to computing the topological degree of  \eqref{shadowsystem}. To find the a priori bound for solutions of \eqref{shadowsystem}, it is inevitable to   encounter with the difficult situation due to the phenomena of collapsing singularities. Indeed, there might be a sequence of  solutions $({w}_k,Q_k)$ of \eqref{shadowsystem} such that $Q_k\notin S_1\cup S_2$ and $Q_k\to Q_0\in S_2$. For the details, we refer to the readers to \cite{llwy,llyz}.

This paper is organized as follows.  In section \ref{sec_sharp}, we put some estimations in order to give a motivation for the construction of approximate solution. In section \ref{sec_approximiate},  we construct an approximate solution. In section  \ref{sec_linear}, we introduce some function spaces and  the linearized operator. In section \ref{sec_invert}, we prove the invertibility  of  linearized operator for the equation \eqref{2.1}. In section \ref{sec_proof}, we finally prove Theorem \ref{theorem1.4}.

\section{Preliminaries}\label{sec_sharp}
In this section, we introduce some estimations from \cite{llty} in order to illustrate the idea of constructing a suitable approximate solution.

Let $u_t$ be a sequence of blow up solutions of \eqref{2.1}, where  $w(x)+8\pi G(x,0)$ is its limit in $C^2_{\textrm{loc}}(M\setminus \{0\})$   and  $w$  satisfies \eqref{thc.equationw}.

Define the local mass $\sigma_0$ of $u_t$ at $0$ by
\begin{equation}
\label{local_mass_at_p}
\sigma_0= \lim_{r\to0}\lim_{t\to0} \frac{\int_{B_{r}(0)}\rho h e^{ {u}_t- G_t^{(2)} }\mathrm{d}v_g}{\int_{M} h e^{ {u}_t- G_t^{(2)} }\mathrm{d}v_g}.
\end{equation}
To understand the blow up phenomena of $u_t$ near $0$, we   consider  the function
\begin{equation}
\begin{aligned}
\label{vt}
{v}_t(y)= u_t(ty) -\ln\left(\int_Mhe^{u_t -  G_t^{(2)}}\mathrm{d}v_g\right)+6\ln  t.
\end{aligned}
\end{equation}
Then ${v}_t(y)$ satisfies the following equation:
\begin{equation}
\begin{aligned}
\label{twosingular}
&\Delta_y {v}_t+\rho{h}(ty)|y-\vec{e}|^2|y+\vec{e}|^2e^{-R_t^{(2)}(ty)}e^{{v}_t(y)}=\rho t^2,
\end{aligned}
\end{equation}
where
\begin{align}
\label{def_of_Rt}
R_t^{(2)}(x)=4\pi R(x,t\vec{e})+4\pi R(x,-t\vec{e}),
\end{align}
and $R(x,\zeta)=G(x,\zeta)+\frac{1}{2\pi}|x-\zeta|$ is the regular part of the Green function.

Let
\begin{equation*}
m_0:=\lim_{R\to+\infty} \lim_{t\to0}\int_{B_{ R}(0)} \rho{h}(ty)|y-\vec{e}|^2|y+\vec{e}|^2e^{-R_t^{(2)}(ty)}e^{{v}_t(y)}\mathrm{d}y.
\end{equation*}
In \cite{llyz}, the following  Pohozaev type identity was derived:
\begin{equation}
\label{pohoidforlocalmass}
\begin{aligned}
(\sigma_0-m_0)(\sigma_0+m_0)=24\pi(\sigma_0-m_0).
\end{aligned}
\end{equation}
Combined with Theorem C, we get $\sigma_0=m_0=8\pi.$ Therefore, the scaled function $v_t$ defined in (\ref{vt}) blows up only at the origin $0$ as $t\to0.$

We set
\begin{align}
\label{3.8}
\widetilde\phi_t(x)= u_t(x)-w(x)-\rho_tG(x,tp_t),
\end{align} where $\rho_t=\frac{\int_{B_{t {R_0}}(tp_t)}\rho  {h}e^{{u}_t-G^{(2)}_t}\mathrm{d}v_g}{\int_M
{h}e^{  {u}_t-G^{(2)}_t}\mathrm{d}v_g}.$
Then our estimation on $\tilde\phi_t$ is stated in the proposition below:
\medskip

\noindent{\bf Proposition 2.A.} \cite{llty} {\em Assume  $\alpha_i\in\mathbb{N}$ for $i\ge 3$, and  $\rho\notin 8\pi\mathbb{N}$. Let $u_t$ be a sequence of blow up solutions of \eqref{2.1}, where  $w(x)+8\pi G(x,0)$ is its limit in $C^2_{\textrm{loc}}(M\setminus \{0\})$. Suppose that $w$ is a non-degenerate solution of \eqref{thc.equationw}. Then for any $0<\tau<1$, there is a constant $c_{\tau}>0$,  independent of $t>0$, satisfying
\begin{align*}
\|\widetilde\phi_t(x)\|_{L^{\infty}(M\setminus B_{2t{R_0}}(tp_t))}\leq c_{\tau} t^{2\tau},\quad
\|\nabla_x\widetilde\phi_t(x)\|_{L^{\infty}(M\setminus B_{2t{R_0}}(tp_t))}\le  c_{\tau} t^{\tau}.
\end{align*}}
We have already known that the scaled function ${v}_t$ blows up only at $0$ in $\mathbb{R}^2$. To give a more description for the behavior of $v_t$ near $0$, we fix a constant  $\mathbf{r}_0\in (0,\frac{1}{2})$, and define
\begin{equation}
\begin{aligned}
\label{2.4}
\lambda_t:=\max_{y\in B_{\mathbf{r}_0}(0)} \bar{v}_t(y)= \bar{v}_t(p_t),\ \textrm{where} \ \ \bar{v}_t(y)=v_t(y)-w(ty),
\end{aligned}
\end{equation}
By using the Pohozaev identity as in \cite[ESTIMATE B]{cl1}, we can derive  the following estimation for the maximum point $p_t$ of $\bar{v}_t$:
\medskip

\noindent {\bf Proposition 2.B.} \cite{llty} {\em Suppose that the assumptions in Proposition 2.A hold. Then there is a constant $c>0$, independent of $t>0$, satisfying
\[|p_t|\le ct.\]}
From Proposition 2.A and Proposition 2.B, we get the blow up solution $u_t$ can be approximated by $w+\rho_t G(x,tp_t)$ well outside a tiny ball which is centered at $0$ and its maximum point $p_t$ is sufficiently close to $0$. Then the left issue is to understand the blow up rate $\lambda_t$. In order to get a estimation for $\lambda_t,$ we have to find out the difference between $u_t$ and the standard bubble in the tiny ball $B_{2R_0t}(tp_t)$. Following the arguments in   \cite{cl1}, we set
\begin{align}
\label{3.17}
I_t(y)=\ln \frac{e^{\lambda_t}}{(1+C_te^{\lambda_t}|y-q_t|^2)^2},
\end{align}
where
\begin{align}
\label{2.4another2}
C_t:=\frac{\rho h (tp_t)e^{w(tp_t)}|p_t-\vec{e}|^2|p_t+\vec{e}|^2e^{-R_t^{(2)}(tp_t)}}{8},
\end{align}
and $q_t$ satisfies that
\begin{align}
\label{3.18}
\nabla_yI_t(y)\Big|_{y=p_t}=-t\rho_t\nabla_xR(tp_t,x)\Big|_{x=tp_t},~|q_t-p_t|\ll1.
\end{align}
It is not difficult to see
\begin{align}
\label{3.19}
|p_t-q_t|=O(te^{-\lambda_t}), \  \textrm{and}\ |I_t(p_t)-\lambda_t|=O(t^2e^{-\lambda_t}).
\end{align}
Let the error term $\eta_t(y)$ in $B_{2R_0}(p_t)$ be defined by
\begin{align}
\label{3.20}
\eta_t(y)=\bar{v}_t(y)-I_t(y)-\rho_t(R(ty,tp_t)-R(tp_t,tp_t))\ \ \textrm{for}\ y\in B_{2R_0}(p_t).
\end{align}
Then \eqref{3.19} and \eqref{3.18} yield
\begin{align}
\label{3.21}
\eta_t(p_t)=\bar{v}_t(p_t)-I_t(p_t)=O(t^2e^{-\lambda_t})~\mathrm{and}~\nabla\eta_t(p_t)=0.
\end{align}
By performing a scaling $y=R_t^{-1}z+p_t$, we define
\begin{align}
\label{tilde_eta}
\t\eta_t(z)=\eta_t(R_t^{-1}z+p_t)~\mathrm{for}~|z|\leq 2R_tR_0, \quad R_t=C_t^{\frac12}e^{\frac{\lambda_t}{2}}.
\end{align}
Notice that $\t\eta_t(z)$ is scaled twice from the original coordinate in a neighborhood of $0\in M$. The reason for us to do this double scaling is that $\bar{v}_t(y)$ also blows up at $0$. Applying the arguments in \cite{cl1}, we get the following result:
\medskip

\noindent{\bf Proposition 2.C.}
\cite{llty} {\em Suppose that the assumptions in Proposition 2.A hold. Then for any $\e\in(0,\frac12)$, there exists a constant $C_{\e}$, independent of $t>0$ and $z\in B_{2R_tR_0}(0)$ such that
\begin{equation*}
|\t\eta_t(z)|\leq C_{\e}(t\|\widetilde\phi_t\|_*+t^{2}|\ln t|)(1+|z|)^{\e}~\mathrm{for}~|z|\leq 2R_tR_0,
\end{equation*}
where
$\|\t\phi_t\|_*=\|\t\phi_t\|_{C^1(M\setminus B_{2tR_0}(tp_t))}.$}
\medskip

Notice that
\begin{equation}
\begin{aligned}
\label{byeta}
\rho\frac{h(\zeta)e^{u_t(\zeta)-G_t^{(2)}(\zeta)}}{\int_M  he^{  u_t-G_t^{(2)}}\mathrm{d}v_g}d\zeta
=8C_te^{I_t(y)}(1+\mathcal{H}_t(y,\eta_t))dy,
\end{aligned}
\end{equation}
where $\zeta=ty$,
\begin{equation}
\label{def_of_htys}
\begin{aligned}
\mathcal{H}_t(y,s)&=\frac{{h}(ty)e^{w(ty)}|y-\vec{e}|^2|y+\vec{e}|^2e^{-R_t^{(2)}(ty)}}{  h(tp_t)e^{w(tp_t)}|p_t-\vec{e}|^2|p_t+\vec{e}|^2e^{-R_t^{(2)}(tp_t)}}
e^{s+\rho_t(R(ty,tp_t)-R(tp_t,tp_t))}-1\\
&=\frac{\rho {h}(ty)e^{w(ty)}|y-\vec{e}|^2|y+\vec{e}|^2e^{-R_t^{(2)}(ty)}}{8C_t}
e^{s+\rho_t(R(ty,tp_t)-R(tp_t,tp_t))}-1,
\end{aligned}
\end{equation}
and
\begin{equation}
\label{def_of_htys2}
\begin{aligned}
\mathcal{H}_t(y,\eta_t) =\mathcal{H}_t(y,s)\Big|_{s=\eta_t(y)}.
\end{aligned}
\end{equation}
After a straightforward computation, we can see that $\mathcal{H}_t(y,\eta_t)$ admits the following expansion,
\begin{equation}
\label{nn3.25}
\begin{aligned}
\mathcal{H}_t(y,\eta_t)&=\mathcal{H}_t(y,0)+\mathcal{H}_t(y,0)\eta_t+O(1)(|\eta_t|)\ \ \textrm{for}\ \ y\in B_{2R_0}(p_t).
\end{aligned}
\end{equation}
Together with \eqref{byeta} and Proposition 2.C, we have the following result:
\medskip

\noindent{\bf Proposition 2.D.}
\cite{llty} {\em
Suppose that  the assumptions in Proposition 2.A hold. Then   there is a constant $c>0$ which is independent of $t$ such that}
\begin{align*}
\left|\rho_t-8\pi \right|\le c(t\|\t\phi_t\|_*+t^{2}|\ln t|).
\end{align*}

From \eqref{vt}, \eqref{3.8} and \eqref{3.20}, we see that for $y\in\partial B_{2R_0}(p_t)$,
\begin{equation*}
\begin{aligned}
\eta_t(y)=~&\t\phi_t(ty)+(\rho_t-8\pi)G(ty,tp_t)+4\ln |y-q_t|-4\ln |y-p_t|\\
&+\lambda_t+2\ln C_t+2\ln  t +8\pi R(tp_t,tp_t)-\ln \int_Mhe^{u_t-G_t^{(2)}}\mathrm{d}v_g\\
&-\ln \frac{(C_te^{\lambda_t}|y-q_t|^2)^2}{(1+C_te^{\lambda_t}|y-q_t|^2)^2}+(8\pi-\rho_t)(R(ty,tp_t)-R(tp_t,tp_t)).
\end{aligned}
\end{equation*}
Combined with Proposition 2.A-2.D, we have the following result:
\medskip

\noindent{\bf Proposition 2.E.}
\cite{llty} {\em
Suppose that  the assumptions in Proposition 2.A hold. Then for any $0<\tau<1$, there is a constant $c_{\tau}>0$ which is independent of $t$ such that}
\begin{align*}
\left|\lambda_t+2\ln  t+2\ln  C_t+8\pi R(tp_t,tp_t)-\ln \left(\frac{\rho}{\rho-8\pi}\int_Mhe^w\right)\right|\le c_{\tau}t^{2\tau}.
\end{align*}

\begin{remark}\label{derive_app}
Proposition 2.A-Proposition 2.E provide us almost all the information for the construction of the blow up solutions $u_t$ of \eqref{2.1}. Precisely, by Proposition 2.A and Proposition 2.D, we need to make the approximate solution $U_{t,q}$ admit the following behavior
\begin{equation}
\label{in}
U_{t,q}(x)=w(x)+8\pi G(x,tp_t)+o(1)\ \ \textrm{in}\ \ M\setminus B_{2tR_0}(tp_t).
\end{equation}
While from Proposition 2.C and \eqref{3.19}, the approximate solution $U_{t,q}$ should satisfy
\begin{equation}
\begin{aligned}
U_{t,q}(x)=~&\ln \frac{e^{\lambda_t}}{t^6(1+\frac{C_te^{\lambda_t}}{t^2}|x-tp_t|^2)^2}+w(x)+\ln \int_Mhe^{u_t-G_t^{(2)}}\mathrm{d}v_g
\\&+8\pi(R(x,tp_t)-R(tp_t,tp_t))+o(1)\ \ \  \ \textrm{in}\ \ B_{2tR_0}(tp_t).
\label{out}
\end{aligned}
\end{equation}
Next, the blow up rate $\lambda_t$ given in Proposition 2.E could help us to well combine \eqref{in} and \eqref{out}. Finally, following the arguments in \cite{ly1}, we make some small modification on such function. Then the modified function becomes our approximate solution and it is in $C^1(M).$ See section \ref{sec_approximiate} for the exact form of the approximate solution for \eqref{2.1}.
\end{remark}
\begin{remark}
\label{orderoft}
In \cite{llty}, the estimations in this section were proved even for a general cases including $\alpha_1=\alpha_2=1$. Notice that the general cases have a slightly different order for $t$,   due  to the setting of $\alpha_1$ and $\alpha_2$.
\end{remark}

\section{Approximate solutions}\label{sec_approximiate}
Let $w$ be a non-degenerate solution of \eqref{thc.equationw}. Note that \eqref{thc.equationw} is invariant by adding a constant to the solutions. Therefore, we
may assume that \begin{equation}\label{mean0}\int_M w\mathrm{d}v_g  =0.\end{equation}
Throughout section \ref{sec_approximiate}-section \ref{sec_proof}, we use $O(1)$ to mean uniform boundedness, independent of $t>0$ and $q$, and  fix  some constants $\mathbf{r}_0\in(0,\frac{1}{2})$ and $R_0>2$.

For any $q\in B_{\mathbf{r}_0}(0)$, and $t>0$, we define
\begin{equation}
\label{def_of_h}
H_{t,q}(y):=h(ty)|y-\vec{e}|^2|y+\vec{e}|^2e^{-R_t^{(2)}(ty)+8\pi R(ty, tq)-8\pi R(tq,tq)+w(ty)-w(tq)},
\end{equation}
\begin{equation}
\begin{aligned}
\label{def_of_lambda}
\lambda_{t,q}:=-2\ln t-2\ln\left(\frac{\rho H_{t,q}(q)}{8}\right)-8\pi R(tq, tq)-w(tq)+\ln\left(\frac{\rho}{\rho-8\pi}\int_Mhe^{w}\right).
\end{aligned}
\end{equation}
To simplify our notation, we set
\begin{align}
\label{defo_of_Ctq}
C_{t,q}:=\sqrt{\frac{8}{\rho H_{t,q}(q)}},\ \ \
\Lambda_{t,q}:=C_{t,q}^{-1}t^{-1}e^{\frac{\lambda_{t,q}}{2}},\ \ \
\Gamma_{t,q}:=C_{t,q}^{-1}e^{\frac{\lambda_{t,q}}{2}}R_0.
\end{align}
By \eqref{def_of_lambda}, we note that \begin{equation}\label{mag}\Lambda_{t,q}=O(t^{-2}) \ \ \textrm{and}\ \
\Gamma_{t,q}=O(t^{-1}).\end{equation}
The function $H_{t,q}$ and $\lambda_{t,q}$ are motivated by \cite{cl2}.  Clearly, $H_{t,q}$ is related to $\mathcal{H}_t(y,0)$ (see \eqref{def_of_htys}) and $\lambda_{t,q}$ is related to the height of the bubbling solutions $\bar{v}_t$ (see \eqref{2.4} and Proposition 2.E). Using Remark \ref{derive_app} with the  modification of  \cite{ly1}, we set
\begin{equation}
\label{component_of_approximate}
u_{t,q}^*(x):=\left\{\begin{array}{l}
\ln\frac{e^{\lambda_{t,q}}}{t^6(1+\Lambda_{t,q}^2|x-tq|^2)^2}+8\pi R(x,tq)(1-\theta_{t,q})-8\pi R(tq,tq)\\
+w(x)-w(tq)+\ln\left(\frac{\rho}{\rho-8\pi}\int_Mhe^w\mathrm{d}v_g\right)\ \mbox{on}\ \ B_{t{R_0}}(tq), \\ \\
\ln\frac{e^{\lambda_{t,q}}}{t^6(1+\Gamma_{t,q}^2)^2}+8\pi \left(G(x,tq)+\frac{1}{2\pi}\ln|t{R_0}|\right)(1-\theta_{t,q})-8\pi R(tq,tq)\\
+w(x)-w(tq)+\ln\left(\frac{\rho}{\rho-8\pi}\int_Mhe^w\mathrm{d}v_g\right)\  \mbox{on}\ \ M\setminus B_{t{R_0}}(tq),
\end{array}\right.
\end{equation}
where $\theta_{t,q}$ is chosen as
\begin{equation}
\label{def_of_theta}
\theta_{t,q}:=\frac{1}{1+\Gamma_{t,q}^2}=O(t^2).
\end{equation}
Then $u_{t,q}^*\in C^1(M)$.  Now we define an approximate solution $U_{t,q}$ for \eqref{2.1} by
\begin{equation}
\label{approximate_sol}
U_{t,q}(x):=u_{t,q}^*(x)-\int_Mu_{t,q}^*\mathrm{d}v_g.
\end{equation}
At first sight, the expression of \eqref{component_of_approximate} seems complicated, but the following result will simplify the expression in \eqref{component_of_approximate} for $x\notin B_{tR_0}(tq)$.
 \begin{lemma}
\label{basic_est_for_app}
(i) $\int_M u_{t,q}^*\mathrm{d}v_g=O(t^2|\ln t|)$.

\noindent (ii) $U_{t,q}(x)=w(x)+8\pi G(x,tq)+O(t^2|\ln t|)$ on $M\setminus B_{t{R_0}}(tq)$.

\noindent (iii) \begin{equation*}
\frac{he^{U_{t,q}-G_t^{(2)}}}{\int_Mhe^{U_{t,q}-G_t^{(2)}}\mathrm{d}v_g}
=\left\{\begin{array}{l}
\frac{\frac{e^{\lambda_{t,q}}}{t^{2}}H_{t,q}\left(\frac{x}{t}\right)(1+O(t^2))}{(1+\Lambda_{t,q}^2|x-tq|^2)^2(1+\mathfrak{A}_{t,q})}\ \ \mbox{on}\ \ B_{t{R_0}}(tq),\\  \\
\left(\frac{\rho-8\pi }{\rho}\right)\frac{he^{w}}{\int_Mhe^w\mathrm{d}v_g}\\
+O(1)\left(1_{B_{\mathbf{r}_0}(0)}(x)\Big(\frac{t|q|}{|x-tq|}+ \frac{t^2}{|x-tq|^2}\Big)+t^2|\ln t|+t|q|\right)\\ \   \ \mbox{on}\ \ M\setminus B_{t{R_0}}(tq),
\end{array}\right.
\end{equation*}
where  $\mathfrak{A}_{t,q}$ is a constant satisfying $\mathfrak{A}_{t,q}=O(t|q|)+O(t^2|\ln t|)$,    $
1_{B_{\mathbf{r}_0}(0)}(x)
=1$ if $x\in B_{\mathbf{r}_0}(0)$,  and $1_{B_{\mathbf{r}_0}(0)}(x)=
0$ if $x\notin B_{\mathbf{r}_0}(0).$
 \end{lemma}
\begin{proof}
\textit{Step 1}. Let
\begin{align}
\label{b_tq}
\mathfrak{B}_{t,q}=-4(\ln|t{R_0}|)\theta_{t,q}+2\ln\Big(\frac{\Gamma_{t,q}^2}{1+\Gamma_{t,q}^2}\Big)=O(t^2|\ln t|).
\end{align}
By the definition of  $u_{t,q}^*$, we see that
\begin{align}
\label{wtq1}
u_{t,q}^*(x)=~&w(x)+8\pi G(x,tq) (1-\theta_{t,q})+2\ln\Big(\frac{\Gamma_{t,q}^2}{1+\Gamma_{t,q}^2}\Big)-4(\ln|t{R_0}|)\theta_{t,q}\nonumber\\
=~&w(x)+8\pi G(x,tq) (1-\theta_{t,q})+\mathfrak{B}_{t,q}\ \ \ \  \textrm{in}\ \  M\setminus B_{t{R_0}}(tq),
\end{align}
where we used \eqref{def_of_lambda}. Similarly, we also see that
\begin{equation}
\begin{aligned}
\label{wtq2}
u_{t,q}^*(x)&=w(x)+8\pi G(x,tq)(1-\theta_{t,q})\\&+2\ln\left(\frac{\Lambda_{t,q}^2|x-tq|^2}{1+\Lambda_{t,q}^2|x-tq|^2}\right)-4\ln|x-tq|\theta_{t,q}\ \ \ \ \textrm{in} \ \  B_{t{R_0}}(tq).
\end{aligned}
\end{equation}
Together with $\int_M w\mathrm{d}v_g=\int_MG(x,tq)\mathrm{d}v_g=0$, we see that
\begin{equation}
\begin{aligned}
\label{wtq3}
&\int_Mu_{t,q}^*\mathrm{d}v_g=\left(\int_{M\setminus B_{t{R_0}}(tq)} +\int_{B_{t{R_0}}(tq)}\right)u_{t,q}^*\mathrm{d}v_g\\
&=\int_{M\setminus B_{t{R_0}}(tq)} w(y)+8\pi G(x,tq) (1-\theta_{t,q}) +\mathfrak{B}_{t,q}\mathrm{d}v_g+\int_{B_{t{R_0}}(tq)} u_{t,q}^*\mathrm{d}v_g\\
&=\int_{M\setminus B_{t{R_0}}(tq)}\mathfrak{B}_{t,q}\mathrm{d}v_g+\int_{B_{t{R_0}}(tq)}u_{t,q}^*-w(y)-8\pi G(x,tq)(1-\theta_{t,q})\mathrm{d}v_g.\\
&=\mathfrak{B}_{t,q}+\int_{B_{t{R_0}}(tq)}\left(2\ln\Big(\frac{\Lambda_{t,q}^2|x-tq|^2}{1+\Lambda_{t,q}^2|x-tq|^2}\Big)
-4\theta_{t,q}\ln {|x-tq|}- \mathfrak{B}_{t,q}\right)\mathrm{d}v_g\\
&=O(t^2|\ln t|).
\end{aligned}
\end{equation}
This proves Lemma \ref{basic_est_for_app}-(i). Moreover, combined with \eqref{b_tq}-\eqref{wtq1} and \eqref{wtq3}, we get Lemma \ref{basic_est_for_app}-(ii).

\medskip
\noindent\textit{Step 2.} From \eqref{wtq1}, we have
\begin{equation}
\begin{aligned}
\label{int_out}
&\int_{M\setminus B_{t{R_0}}(tq)} he^{u_{t,q}^*-G_t^{(2)}}\mathrm{d}v_g\\
&=e^{\mathfrak{B}_{t,q}}\int_{M\setminus B_{t{R_0}}(tq)}he^{w+8\pi G(x,tq)(1-\theta_{t,q})-4\pi G(x,-t\vec{e})-4\pi G(x,t\vec{e})}\mathrm{d}v_g\\
&=\int_{M\setminus B_{t{R_0}}(tq)}he^w
\left(1+O(1)\Big(1_{B_{\mathbf{r}_0}(0)}\Big(\frac{t|q|}{|x-tq|}+ \frac{t^2}{|x-tq|^2}\Big)+t^2|\ln t|+t|q|\Big)\right)\mathrm{d}v_g\\
&=\int_{M}h(x)e^{ w(x)}\mathrm{d}v_g+ O(1)\left(t|q| +t^2|\ln t|\right).
\end{aligned}
\end{equation}
Using \eqref{component_of_approximate}, we get
\begin{equation}
\begin{aligned}
\label{int_inside}
&(\rho-8\pi)\int_{B_{t{R_0}}(tq)} h(x)e^{u_{t,q}^*-G_t^{(2)}(x)}\mathrm{d}v_g\\
&=\rho\int_Mhe^w\mathrm{d}v_g
\int_{B_{\Gamma_{t,q}}(0)}\frac{8H_{t,q}(C_{t,q}e^{-\frac{\lambda_{t,q}}{2}}z+q) e^{-8\pi R(\Lambda_{t,q}^{-1}z+tq,tq)\theta_{t,q}}}{\rho H_{t,q}(q)(1+|z|^2)^2}\mathrm{d}z\\
&=\int_Mhe^w\mathrm{d}v_g\int_{B_{\Gamma_{t,q}}(0)}
\frac{8 (1+\frac{\nabla H_{t,q}(q)}{H_{t,q}(q)}\cdot (C_{t,q}e^{-\frac{\lambda_{t,q}}{2}}z)+O(1)(e^{-\lambda_{t,q}}|z|^2+t^2))}{(1+|z|^2)^2}\mathrm{d}z\\
&=8\pi\int_Mhe^w\mathrm{d}v_g+O(t^2|\ln t|).
\end{aligned}
\end{equation}
From \eqref{int_out}-\eqref{int_inside}, we obtain
\begin{equation}
\begin{aligned}
\label{int_total}
\int_{M} h(x)e^{u_{t,q}^*-G_t^{(2)}(x)}\mathrm{d}v_g=\left(\frac{\rho}{\rho-8\pi}\int_Mhe^w\mathrm{d}v_g\right)(1+\mathfrak{A}_{t,q}),
\end{aligned}
\end{equation}
where $\mathfrak{A}_{t,q}=O(1)(t|q|+t^2|\ln t|)$ is a constant.

Note that $$\frac{h(x)e^{U_{t,q}(x)-G_t^{(2)}(x)}}{\int_Mhe^{U_{t,q}-G_t^{(2)}}\mathrm{d}v_g}
=\frac{h(x)e^{u_{t,q}^*(x)-G_t^{(2)}(x)}}{\int_Mhe^{u_{t,q}^*-G_t^{(2)}}\mathrm{d}v_g}.$$
Together with the definition of $u_{t,q}^*$, \eqref{wtq1}, and \eqref{int_total}, we get the left conclusion of Lemma \ref{basic_est_for_app} and it finishes the proof.
\end{proof}

\section{Linearized operator}\label{sec_linear}
\subsection{Function spaces}
\label{linearizedoperator}
Let $\chi_{t,q}$ be the cut-off function satisfying
\begin{equation}
\label{cutoff}
0\le \chi_{t,q} \le 1,\quad |\nabla_x \chi_{t,q}(x)|=O(t^{-1}),\quad   |\nabla^2_x \chi_{t,q}(x)|=O(t^{-2}),
\end{equation}
\begin{equation}
\label{definitionofcutoff}
\chi_{t,q}(x)=\chi_{t,q}(|x-tq|)=
\left\{\begin{array}{l}
1\ \ \mbox{on}\ \ B_{\frac{t{R_0}}{2}}(tq), \\
0\ \ \mbox{on}\ \ M\setminus B_{t{R_0}}(tq),
\end{array}\right.
\end{equation}
We denote $z=(z_1,z_2)\in\mathbb{R}^2$, and set
\begin{equation}
\label{def_rho}\rho(z):=(1+|z|)^{-1-\frac{\alpha}{2}}.
\end{equation}
We recall the function $Y_i(z)$, $i=0,1,2$, defined in \eqref{z01}.
We   set for $i=1,2$,
\begin{align}
&\hat{Y}_{t,q,i}(x)=\chi_{t,q}(x)Y_i\left(\Lambda_{t,q}(x-tq)\right),\label{hatyt}\\
&Z_{t,q,i}(x):=-\Delta_x\hat{Y}_{t,q,i}(x)+\frac{8\Lambda_{t,q}^2\chi_{t,q}(x)\hat{Y}_{t,q,i}(x)}{(1+\Lambda_{t,q}^2|x-tq|^2)^2}.\label{z1}
\end{align}
In this section, we often use $x$ as the original coordinate in a neighborhood of $0\in M$ and $z=\Lambda_{t,q}(x-tq)$ as the double scaled coordinate. For convenience, we write
\begin{equation}
\label{scaledoverline}
\overline{f}(z)=f(\Lambda_{t,q}^{-1}z+tq).
\end{equation}
Using the notion above, we have
\begin{equation}
\label{definitionofcutoff2}
\overline{\chi_{t,q}}(z)=
\overline{\chi_{t,q}}(|z|)=
\left\{\begin{array}{l}
1\ \ \mbox{on}\ \ B_{\Gamma_{t,q}/2}(0),\\
0\ \ \mbox{on}\ \ \mathbb{R}^2\setminus B_{\Gamma_{t,q}}(0),
\end{array}\right.
\quad 0\le \overline{\chi_{t,q}} \le 1,
\end{equation}
and $|\nabla_z \overline{\chi_{t,q}}(z)|=O(t)$, $|\nabla^2_z \overline{\chi_{t,q}}(z)|=O(t^{2}).$
\medskip

Concerning for $Z_{t,q,i}(x)$, we have the following result.
\begin{lemma}
\label{ztqi}
(i) $\int_M Z_{t,q,i}(x)  \mathrm{d}v_g=0$ for $i=1, 2$.

(ii) For $i=1, 2$,
\begin{equation*}
\begin{aligned}
\int_M Z_{t,q,i}(x) \hat{Y}_{t,q,i}(x)\mathrm{d}v_g
&=\Lambda_{t,q}^{-2}\int_{B_{\Gamma_{t,q}}(0)}  \overline{Z_{t,q,i}}(z)\overline{\chi_{t,q}}(z) Y_i(z)\mathrm{d}z\\
&=\int_{B_{\Gamma_{t,q}}(0)}|\nabla_z(\overline{\chi_{t,q}}(z)Y_i(z))|^2+\frac{8(\overline{\chi_{t,q}}(z))^3(Y_i(z))^2}{(1+|z|^2)^2}\mathrm{d}z.
\end{aligned}
\end{equation*}

\noindent (iii) $\int_M Z_{t,q,i}(x) \hat{Y}_{t,q,j}(x)\mathrm{d}x
=\Lambda_{t,q}^{-2}\int_{B_{\Gamma_{t,q}}(0)}\overline{Z_{t,q,i}}(z)\overline{\chi_{t,q}}(z)Y_j(z)\mathrm{d}z=0$ for $i\neq j$.
\end{lemma}
\begin{proof}
(i) From the definition of $Z_{t,q,i}(x)$ and $\hat{Y}_{t,q,i}(x)$, we have for  $i=1, 2$,
\begin{equation}
\begin{aligned}
\label{haty}
&\overline{Z_{t,q,i}}(z) = Z_{t,q,i}(\Lambda_{t,q}^{-1}z+tq)\\
&=-\Delta_x\hat{Y}_{t,q,i}(x)\Big|_{x=\Lambda_{t,q}^{-1}z+tq} +\frac{8\Lambda_{t,q}^2\chi_{t,q}(x)\hat{Y}_{t,q,i}(x)}{(1+\Lambda_{t,q}^2|x-tq|^2)^2}\left|_{x=\Lambda_{t,q}^{-1}z+tq}\right.\\
&=-\Delta_x \left(\chi_{t,q}(x)Y_i\Big(\Lambda_{t,q}(x-tq)\Big)\right)\Big|_{x=\Lambda_{t,q}^{-1}z+tq}\\
&\quad+\frac{8\Lambda_{t,q}^2\chi_{t,q}^2(x)Y_i\Big(\Lambda_{t,q}(x-tq)\Big)}{(1+\Lambda_{t,q}^2|x-tq|^2)^2}\left|_{x=\Lambda_{t,q}^{-1}z+tq}\right.\\
&=\Lambda_{t,q}^2\Big(-\Delta_z (\overline{\chi_{t,q}}(z)Y_i(z))+\frac{  8 \overline{\chi_{t,q}}^2(z)Y_i(z)}{(1+ |z|^2)^2}\Big).
\end{aligned}
\end{equation}
After direct computation, we have for $1
\le i\neq j\le 2$,
\begin{equation}
\label{derivofyi}
\frac{\partial Y_i(z)}{\partial z_i}=\frac{ 1+ z_j^2-z_i^2}{(1+|z|^2)^2},\ \ \frac{\partial Y_i(z)}{\partial z_j}
=\frac{ -2z_iz_j}{(1+|z|^2)^2}, \ \ \Delta Y_i(z) = -\frac{8 z_i}{(1+|z|^2)^3},
\end{equation}
and
\begin{align}
\label{by}
&\Delta_z (\overline{\chi_{t,q}}(z)Y_i(z))-\frac{8\overline{\chi_{t,q}}^2(z)Y_i(z)}{(1+ |z|^2)^2}\nonumber\\
&=(\Delta_z \overline{\chi_{t,q}}(|z|))\frac{ z_i }{(1+|z|^2) }+2\frac{d\overline{\chi_{t,q}}(|z|)}{d|z|}\Big(\frac{z_i}{|z|}\frac{(1+ z_j^2- z_i^2)}{(1+|z|^2)^2}-\frac{2z_iz_j^2 }{|z|(1+|z|^2)^2}\Big)\\
&~\quad-\frac{8\overline{\chi_{t,q}}(|z|)(1+\overline{\chi_{t,q}}(|z|))z_i }{(1+|z|^2)^3}.\nonumber
\end{align}
Combined with \eqref{haty}, we get Lemma \ref{ztqi}-(i).
\medskip

\noindent (ii) For the functions $Z_{t,q,i}(x)$ and  $\hat{Y}_{t,q,i}(x)$, we have for $1\le i,j\le 2$,
\begin{equation}
\begin{aligned}
\label{haty1}
\int_M Z_{t,q,i}(x) \hat{Y}_{t,q,j}(x)\mathrm{d}v_g
&=\int_{B_{tR_0}(tq)}Z_{t,q,i}(x)\chi_{t,q}(x)Y_j\Big(\Lambda_{t,q}(x-tq)\Big)\mathrm{d}x\\
&=\Lambda_{t,q}^{-2}\int_{B_{\Gamma_{t,q}}(0)}  \overline{Z_{t,q,i}}(z)\overline{\chi_{t,q}}(z) Y_j(z)\mathrm{d}z.
\end{aligned}
\end{equation}
Note $\overline{\chi_{t,q}}(z)\equiv0$ on $\mathbb{R}^2\setminus B_{\Gamma_{t,q}}(0)$.
Using \eqref{haty} and the integration by parts, we get
\begin{equation}
\begin{aligned}
\label{haty2}
&\int_{B_{\Gamma_{t,q}}(0)} \overline{Z_{t,q,i}}(z)\overline{\chi_{t,q}}(z) Y_i(z)\mathrm{d}z =\int_{\mathbb{R}^2} \overline{Z_{t,q,i}}(z)\overline{\chi_{t,q}}(z) Y_i(z)\mathrm{d}z\\
&=\int_{\mathbb{R}^2}\Lambda_{t,q}^2\left(-\Delta_z (\overline{\chi_{t,q}}(z)Y_i(z))  +\frac{  8 \overline{\chi_{t,q}}^2(z)Y_i(z)}{(1+ |z|^2)^2}\right)\overline{\chi_{t,q}}(z) Y_i(z)\mathrm{d}z\\
&=\Lambda_{t,q}^2\int_{B_{\Gamma_{t,q}}(0)}\left(|\nabla_z(\overline{\chi_{t,q}}(z)Y_i(z))|^2
+\frac{8(\overline{\chi_{t,q}}(z))^3(Y_i(z))^2}{(1+|z|^2)^2}\right)\mathrm{d}z.
\end{aligned}
\end{equation}
Then Lemma \ref{ztqi}-(ii) follows from \eqref{haty1} and \eqref{haty2}.
\medskip

\noindent (iii) We can obtain Lemma \ref{ztqi}-(iii) by \eqref{haty1}, \eqref{haty}, and \eqref{by}.
\end{proof}

Following \cite{cfl}, we will introduce some function spaces with different norms for the different regions $M\setminus B_{\frac{t{R_0}}{2}}(tq)$ and $B_{tR_0}(tq)$.
\begin{definition}
\label{D3}For any $0<\alpha<\frac{1}{2}$ and $p\in(1,2]$,

\noindent (a) A function $\phi(x)$ on $M$ is said to be in ${X_{\alpha,t,p}}$ if $\int_M\phi \mathrm{d}v_g=0$ and
\begin{equation*}
\begin{aligned} \|\phi \|_{X_{\alpha,t,p}}:=~&\|\Delta_z
\overline{\phi}(z)(1+|z|)^{1+\frac{\alpha}{2}}\|_{L^2(B_{\Gamma_{t,q}}(0))}+\|\overline{\phi}(z)\rho(z)\|_{L^2(B_{\Gamma_{t,q}}(0))}
\\&+\|\Delta \phi\|_{L^p(M\setminus B_{\frac{t{R_0}}{2}}(tq))}+\|\phi\|_{L^p(M\setminus B_{\frac{t{R_0}}{2}}(tq))}<+\infty,\end{aligned}
\end{equation*}
where $\Delta_z
\overline{\phi}(z)=\sum_{i=1}^2\frac{\partial^2\overline{\phi}(z)}{\partial z_i^2}$.

\noindent (b) A function $g(x)$ on $M$ is said to be in $Y_{\alpha,t,p}$ if $\int_M g\mathrm{d}v_g=0$ and
\begin{equation*}
\begin{aligned}\|g \|_{Y_{\alpha,t,p}}:&=\Big\|t^2e^{-\lambda_{t,q}}\overline{g}(z)(1+|z|)^{1+\frac{\alpha}{2}}\Big\|_{L^2(B_{\Gamma_{t,q}}(0))} +\|g\|_{L^p(M\setminus B_{\frac{t{R_0}}{2}}(tq))}<+\infty.\end{aligned}
\end{equation*}

\noindent  (c) $E_{\alpha,t,q,p}:= \left\{\phi \in {X_{\alpha,t,p}}  \ \Big|\ \int_{M} \phi(x) Z_{t,q,i}(x) \mathrm{d}v_g=0,\ i=1,2\right\}.$

\noindent  (d) $F_{\alpha,t,q,p}: =\Big\{g\in Y_{\alpha,t,p}  \ \Big|\  \int_{M} g(x) \hat{Y}_{t,q,i}(x)\mathrm{d}v_g=0,  \ i=1,2\Big\}.$
\end{definition}

Note we can get $X_{\alpha,t,p}, Y_{\alpha,t,p}\subseteq L^1(M)$ from H\"{o}lder inequality, even though $B_{\Gamma_{t,q}}(0)$ is not uniformly bounded. Furthermore, we have the following result.
\begin{lemma}
\label{bddchi2}
$\|Z_{t,q,i}\|_{Y_{\alpha,t,p}}=O(1)$.
\end{lemma}
\begin{proof}
By \eqref{haty}, we   see that
\begin{equation}
\begin{aligned}\label{ztqibdd}
&\|Z_{t,q,i}\|_{Y_{\alpha,t,p}}=\Big\|t^2e^{-\lambda_{t,q}}\overline{Z_{t,q,i}}(z)(1+|z|)^{1+\frac{\alpha}{2}}\Big\|_{L^2(B_{\Gamma_{t,q}}(0))} +\|Z_{t,q,i}\|_{L^p(M\setminus B_{\frac{t{R_0}}{2}}(tq))}\\
&\le O(1)\Bigg(  \Big\|\Big\{-\Delta_z (\overline{\chi_{t,q}}(z)Y_i(z))  +\frac{  8 \overline{\chi_{t,q}}^2(z)Y_i(z)}{(1+ |z|^2)^2}\Big\}(1+|z|)^{1+\frac{\alpha}{2}}\Big\|_{L^2(B_{\Gamma_{t,q}}(0))}\\
&\quad+\left\|-\Delta_x \Big(\chi_{t,q}(x)Y_i\Big(\Lambda_{t,q}(x-tq)\Big)\Big)\right\|_{L^p(B_{tR_0}(tq))\setminus B_{\frac{t{R_0}}{2}}(tq))}\\
&\quad+\left\|\frac{\Lambda_{t,q}^2Y_i\Big(\Lambda_{t,q}(x-tq)\Big)}{(1+\Lambda_{t,q}^2|x-tq|^2)^2}\right\|_{L^p(B_{tR_0}(tq))\setminus B_{\frac{t{R_0}}{2}}(tq))}\Bigg).
\end{aligned}
\end{equation}
Recall that $|\nabla  \overline{\chi_{t,q}} |=O(t)$, $|\nabla^2 \overline{\chi_{t,q}} |=O(t^{2})$ and $p\in(1,2],$ we can obtain the Lemma \ref{bddchi2} from \eqref{ztqibdd} by direct computation.
\end{proof}

\subsection{Projection and Linearized operator}
We define the projection $\mathbb{Q}_{t,q}:Y_{\alpha,t,p}\rightarrow F_{\alpha,t,q,p}$ by
\[(\mathbb{Q}_{t,q}g)(x):= g(x)-\sum_{i=1}^{2}c_{i}Z_{t,q,i}(x),\ \ \textrm{ where}\  c_i\ \textrm{ is chosen so that}\ \ \mathbb{Q}_{t,q}g\in F_{\alpha,t,q,p}.\]
We shall prove that  the projection  $\mathbb{Q}_{t,q}$ is a well-defined, and bounded map.
\begin{lemma}
\label{pronorm}
(i) $\mathbb{Q}_{t,q}$ is well-defined, that is, for any $g\in  Y_{\alpha,t,p}$, there exists a unique constant $c_{g,i}$, $i=1, 2$ such that
 $g -\sum_{i=1}^{2}c_{g,i}Z_{t,q,i}  \in F_{\alpha,t,q,p}.$

(ii) $\mathbb{Q}_{t,q}$ is bounded, that is, there exists $t_0>0$  such that if $0<t<t_0$, then
\begin{equation*}
\Vert \mathbb{Q}_{t,q}g\Vert_{Y_{\alpha,t,p}}\leq c\Vert g\Vert_{Y_{\alpha,t,p}}
\end{equation*}
for some constant $c>0$, which is independent of $t\in(0,t_0)$ and $g\in Y_{\alpha,t,p}$
\end{lemma}
\begin{proof}
(i) For any $g\in  Y_{\alpha,t,p}$, let
\begin{equation}
\label{ciforh}
\begin{aligned}
c_{g,i}=\frac{\int_{B_{\Gamma_{t,q}}(0)}\Lambda_{t,q}^{-2}\overline{g} (z) \overline{\chi_{t,q}}(z)Y_i(z)\mathrm{d}z}{\int_{B_{\Gamma_{t,q}}(0)} \left(|\nabla_z(\overline{\chi_{t,q}}(z)Y_i(z))|^2+\frac{8(\overline{\chi_{t,q}}(z))^3(Y_i(z))^2}{(1+|z|^2)^2}\right)\mathrm{d}z},\ \   i=1,2.
\end{aligned}
\end{equation}
By   Lemma \ref{ztqi}, we get that for $i=1,2$,
\begin{equation*}
\begin{aligned}
&\int_{M}(g(x)-\sum_{k=1}^{2}c_{g,k}Z_{t,q,k}(x)) \hat{Y}_{t,q,i}(x)\mathrm{d}v_g\\
&=\int_{B_{tR_0}(tq)} g(x)\chi_{t,q}(x)Y_i(\Lambda_{t,q}(x-tq))\mathrm{d}x
-\int_{M}\Big(\sum_{k=1}^{2}c_{g,k}Z_{t,q,k}(x)\Big)\hat{Y}_{t,q,i}(x)\mathrm{d}x\\
&=\int_{B_{\Gamma_{t,q}}(0)}\Lambda_{t,q}^{-2}\overline{g} (z) \overline{\chi_{t,q}}(z)Y_i(z)\mathrm{d}z\\
&\quad-c_{g,i}\int_{B_{\Gamma_{t,q}}(0)}
\left(|\nabla_z(\overline{\chi_{t,q}}(z)Y_i(z))|^2+\frac{8(\overline{\chi_{t,q}}(z))^3(Y_i(z))^2}{(1+|z|^2)^2}\right)\mathrm{d}z.
\end{aligned}\end{equation*}
Using \eqref{ciforh}, we get
\begin{equation}
\label{ciforh2}
\int_{M}(g(x)-\sum_{k=1}^{2}c_{g,k}Z_{t,q,k}(x)) \hat{Y}_{t,q,i}(x)\mathrm{d}v_g=0,\ \  i=1,2.
\end{equation}
From Lemma \ref{ztqi} we have
$$\int_M Z_{t,q,i}  \hat{Y}_{t,q,i} \mathrm{d}v_g>0\quad\mathrm{and}\quad\int_MZ_{t,q,i}  \hat{Y}_{t,q,j} \mathrm{d}v_g=0,~i\neq j.$$
As a consequence, we can uniquely get $c_{g,i}$ such that \eqref{ciforh2} holds.  This proves Lemma \ref{pronorm}-(i).
\medskip

\noindent (ii) Using \eqref{ciforh} and H\"{o}lder inequality,  we get  that
\begin{equation}\label{bddchi1}
\begin{aligned}
|c_{g,i}|&= \frac{\left|\int_{B_{\Gamma_{t,q}}(0)}\Lambda_{t,q}^{-2}\overline{g} (z) \overline{\chi_{t,q}}(z)Y_i(z)\mathrm{d}z\right|}{\int_{B_{\Gamma_{t,q}}(0)} \left(|\nabla_z(\overline{\chi_{t,q}}(z)Y_i(z))|^2+\frac{8(\overline{\chi_{t,q}}(z))^3(Y_i(z))^2}{(1+|z|^2)^2}\right)\mathrm{d}z}\\& \le O(1)(
 \|g\|_{Y_{\alpha,t,p}}\|Y_i(z)(1+|z|)^{-1-\frac{\alpha}{2}}\|_{L^2(\mathbb{R}^2 )})\le O(1)( \|g\|_{Y_{\alpha,t,p}}).
\end{aligned}\end{equation}
By \eqref{bddchi1} and Lemma \ref{bddchi2}, we have
\begin{equation*}
\begin{aligned}
\|\mathbb{Q}_{t,q}g\|_{Y_{\alpha,t,p}}&\le\|g\|_{Y_{\alpha,t,p}}+\sum_{i=1}^2|c_{g,i}|\|Z_{t,q,i}\|_{Y_{\alpha,t,p}}  \le O(1)( \|g\|_{Y_{\alpha,t,p}}),
\end{aligned}
\end{equation*}which implies Lemma \ref{pronorm}-(ii). Thus we finish the proof.
\end{proof}

Let
\begin{equation}
\label{def_of_linearized_operator}
\mathbb{L}_{t,q}\phi:=\Delta \phi\left( x\right) +\rho\frac{h(x)e^{U_{t,q}(x)-G_t^{(2)}(x)}}{\int_Mh e^{U_{t,q} -G_t^{(2)} }\mathrm{d}v_g}\left(\phi(x)-\frac{\int_Mh e^{U_{t,q} -G_t^{(2)} }\phi \mathrm{d}v_g}{\int_Mh e^{U_{t,q} -G_t^{(2)} }\mathrm{d}v_g}\right).
\end{equation}
Then we state the main result in this section.
\begin{theorem}
\label{thma}
There exist $r_0, t_0>0$ such that if $0<t<t_0$ and  $q\in B_{r_0}(0)$, then the map
$\mathbb{Q}_{t,q}\mathbb{L}_{t,q}:E_{\alpha,t,q,p}\rightarrow  F_{\alpha,t,q,p}$
is isomorphism. Moreover, for any $(\phi,g)\in E_{\alpha,t,q,p}\times  F_{\alpha,t,q,p}$ such that
$\mathbb{Q}_{t,q}\mathbb{L}_{t,q}\phi=g,$ the following inequality holds:
\begin{equation}
\label{ineq}
\Vert \phi\Vert_{L^{\infty}(M)}+\|\phi\|_{X_{\alpha,t,p}}\leq C|\ln t|\Vert g\Vert_{Y_{\alpha,t,p}},
\end{equation}
for some constant $C>0$ independent of $(t,q,\phi,g)\in (0,t_0)\times B_{{r}_0}(0)\times
 E_{\alpha,t,q,p}\times  F_{\alpha,t,q,p}$.
\end{theorem}
We will give the proof in section \ref{sec_invert}.
It is remarkable that for the linearized operator  $\mathbb{L}_{t,q}$, the orthogonality conditions shall be considered with only the approximate kernels due to translations, i.e. $Y_1$, $Y_2$. Heuristically, since the dilations part $Y_0$ does not vanish at infinity,  the  non-concentration phenomena  of mass  disturbs the element caused by dilations   to be  a good approximate kernel. It causes the main difficulty to prove Theorem \ref{thma}. This is completely different from the previous works related to concentration phenomena of mass (for example, see \cite{cl2,ly1}).

We note that the norm of $(\mathbb{Q}_{t,q}\mathbb{L}_{t,q})^{-1}$ is $O(|\ln t|)$, which   was  appeared in \cite{cfl} by Chan, Fu, Lin for the study of Chern-Simons-Higgs equation  (see also \cite{dkm,ly1}).

\section{Invertibility of Linearized operator}\label{sec_invert}
First of all, we want to show  the inequality \eqref{ineq}. We shall prove it by contradiction. Suppose that as $t\to0$, there is a sequence $q_t$,   $(\phi_t,g_{t})\in E_{\alpha,t,q_t,p}\times F_{\alpha,t,q_t,p}$ satisfying $\mathbb{Q}_{t,q_t}\mathbb{L}_{t,q_t}\phi_t=g_{t}, $ and
\begin{equation}
\label{eqto1}
\Vert \phi_t\Vert_{L^{\infty}(M)}+\|\phi_t\|_{X_{\alpha,t,p}}=1,  \quad
\Vert g_{t}\Vert_{Y_{\alpha,t,p}}=o(|\ln t|^{-1}).
\end{equation}
To simplify our notation, we write $q_t$ by $q$.

Since $\phi_t$ satisfies $\mathbb{Q}_{t,q}\mathbb{L}_{t,q}\phi_t=g_{t},$ we could find constants $c_{t,q,i}$, $i=1,2$ such that
\begin{equation}
\begin{aligned}
\label{c}
&\Delta\phi_t+\rho\frac{h(x)e^{U_{t,q}(x)-G_t^{(2)}(x)}}{\int_Mh e^{U_{t,q}-G_t^{(2)}}\mathrm{d}v_g}
\Big(\phi_t(x)-\frac{\int_Mh e^{U_{t,q} -G_t^{(2)} }\phi_t \mathrm{d}v_g}{\int_Mh e^{U_{t,q} -G_t^{(2)} }\mathrm{d}v_g}\Big)
\\&=g_{t}(x)+\sum_{i=1}^2c_{t,q,i}Z_{t,q,i}(x),
\end{aligned}
\end{equation}and
\begin{equation}
\label{intcfind}
\begin{aligned}
&\int_{M}\left[\Delta\phi_t+\rho\frac{h(x)e^{U_{t,q}(x)-G_t^{(2)}(x)}}{\int_Mh e^{U_{t,q} -G_t^{(2)}}\mathrm{d}v_g}
\left(\phi_t(x)-\frac{\int_Mh e^{U_{t,q} -G_t^{(2)} }\phi_t \mathrm{d}v_g}{\int_Mh e^{U_{t,q} -G_t^{(2)}}\mathrm{d}v_g}\right)\right]
\hat{Y}_{t,q,i}(x)\mathrm{d}v_g\\
&=\int_M\left(g_t(x)+\sum_{i=1}^2c_{t,q,i}Z_{t,q,i}(x)\right)\hat{Y}_{t,q,j}(x)\mathrm{d}v_g\ \textrm{for}\ j=1, 2.
\end{aligned}
\end{equation}
Because the proof is long and complicated, we would like to first explain the ideas behind all  computations. It is not difficult to see that $\phi_t\to \phi_0$ in $C^{0,\beta}_{\textrm{loc}}(M\setminus\{0\})$ for some $\phi_0$. Then we need to show $\phi_0\equiv0$ to derive a contradiction to \eqref{eqto1}.
To achieve our goal, we need to complete the following three steps,

\medskip
\begin{enumerate}
  \item[(i).] The (RHS) of \eqref{c} tends to zero,
  \item[(ii).] $\frac{he^{U_{t,q}-G_t^{(2)}}}{\int_Mhe^{U_{t,q}-G_t^{(2)}}\mathrm{d}v_g}\to\left(\frac{\rho-8\pi }{\rho }\right)\frac{ he^{w}}{\int_Mhe^w\mathrm{d}v_g}$ in $C^0_{\textrm{loc}}(M\setminus\{0\})$,
  \item[(iii).] $\frac{\int_Mhe^{U_{t,q} -G_t^{(2)}}\phi_t\mathrm{d}v_g}{\int_Mhe^{U_{t,q}-G_t^{(2)}}\mathrm{d}v_g}\to
\frac{\int_{M}he^{w}\phi_0\mathrm{d}v_g}{\int_Mhe^{w}\mathrm{d}v_g}$.
\end{enumerate}

%
%
%

\noindent
Once we get (i)-(iii), we can show that equation \eqref{c} converges to
\[\Delta \phi_0+(\rho-8\pi)\frac{he^{w}}{\int_M he^w\mathrm{d}v_g}\left(\phi_0-\frac{\int_Mhe^{w}\phi_0\mathrm{d}v_g}{\int_M he^w\mathrm{d}v_g}\right)=0.\]
Then from the non-degenerate condition on $w$, we can get $\phi_0\equiv0$. Among the steps (i)-(iii), (ii) was already proved (see Lemma \ref{basic_est_for_app}-(iii)) and (i) is not to difficult. However, the proof of (iii) is very difficult and complicate. To verify it, we have to prove some identities resulting from the bubbling behavior at $0$.

For the equation \eqref{c} in the tiny ball $B_{tR_0}(tq)$,  we apply the doubly scaling as in \eqref{scaledoverline} and define
\begin{equation}
\label{defo_of_psi_t}
\psi_t(z)=\overline{\phi}_t(z)-\frac{\int_Mh e^{U_{t,q} -G_t^{(2)} }\phi_t \mathrm{d}v_g}{\int_Mh e^{U_{t,q} -G_t^{(2)} }\mathrm{d}v_g}.
\end{equation}
By \eqref{defo_of_psi_t} and \eqref{eqto1}, we have
\begin{equation}
\label{bdd_of_psi_t}
\|\psi_t\|_{L^{\infty}(B_{\Gamma_{t,q}}(0))}\le 2 \|{\phi}_t\|_{L^\infty(M)} \le 2.
\end{equation}
We note that
\begin{equation}
\label{laplace_of_psi_t}
\Delta_z\psi_t(z)=\Delta_z\phi_t(\Lambda_{t,q}^{-1}z+tq)=\Lambda_{t,q}^{-2}\Big(\Delta_x\phi_t(x)\Big|_{x=\Lambda_{t,q}^{-1}z+tq }\Big).
\end{equation}
By Lemma \ref{basic_est_for_app}-(iii), we   see that $\psi_t$ satisfies
\begin{equation}
\label{eq_for_psi}
\begin{aligned}
&\Delta_z\psi_t(z)+\frac{8\psi_t(z)\frac{H_{t,q}(C_{t,q}e^{-\frac{\lambda_{t,q}}{2}}z+q)}{H_{t,q}(q)}(1+O(1)(t|q|+t^2|\ln t|))}{(1+|z|^2)^2}
\\&=\Lambda_{t,q}^{-2}\Big(\overline{g_{t}}(z)+\sum_{i=1}^2c_{t,q,i} \overline{Z_{t,q,i}}(z)\Big)\ \mbox{on}\ B_{\Gamma_{t,q}}(0).
\end{aligned}
\end{equation}
In the following lemma, we give an estimation on $c_{t,q,i}$, $i=1,2$
\begin{lemma}
\label{magofcc}
$|c_{t,q,i}|\|Z_{t,q,i}\|_{Y_{\alpha,t,p}}=O(t)$  for  $i=1,2$.
\end{lemma}
\begin{proof}
Multiplying $\overline{\chi_{t,q}}(z)Y_j(z)$ on both sides of \eqref{eq_for_psi} and using Lemma \ref{ztqi}, we have
\begin{equation*}
\begin{aligned}
&c_{t,q,j}\int_{B_{\Gamma_{t,q}}(0)} |\nabla_z(\overline{\chi_{t,q}}(z)Y_j(z))|^2+\frac{8(\overline{\chi_{t,q}}(z))^3(Y_j(z))^2}{(1+|z|^2)^2} \mathrm{d}z\\
&=\int_{B_{\Gamma_{t,q}}(0)}\Big\{\Delta\psi_t+\frac{8\psi_t(z)}{(1+|z|^2)^2}
+\frac{ O(1)(|\psi_t(z)|(e^{-\frac{\lambda_{t,q}}{2}}|z|+t|q|+t^2|\ln t|))}{(1+|z|^2)^2}\\
&\quad- \Lambda^{-2}_{t,q}\overline{g_{t}}(z)\Big\}\frac{\overline{\chi_{t,q}}(z) z_j }{(1+|z|^2)}\mathrm{d}z.
\end{aligned}
\end{equation*}
We note that $\int_{B_{\Gamma_{t,q}}(0)}\overline{g_{t}}(z)\frac{\overline{\chi_{t,q}}(z) z_j }{(1+|z|^2)}\mathrm{d}z=0$ from $g_{t}\in F_{\alpha,t,q,p}$.
After integration by parts, we have
\begin{equation*}
\begin{aligned}
&c_{t,q,j}\int_{B_{\Gamma_{t,q}}(0)} |\nabla_z(\overline{\chi_{t,q}}(z)Y_j(z))|^2+\frac{8(\overline{\chi_{t,q}}(z))^3(Y_j(z))^2}{(1+|z|^2)^2} \mathrm{d}z\\
&=\int_{B_{\Gamma_{t,q}}(0)\setminus B_{\frac{\Gamma_{t,q}}{2}}(0)}\psi_t(z)\Big\{\Delta\overline{\chi_{t,q}}(z)\frac{ z_j }{(1+|z|^2)} +2\nabla\overline{\chi_{t,q}}\cdot\nabla\Big(\frac{ z_j }{(1+|z|^2)}\Big)\Big\}\mathrm{d}z\\
&\quad+\int_{B_{\Gamma_{t,q}}(0)}\psi_t(z)\Big\{ \overline{\chi_{t,q}}(z)\Delta\Big(\frac{z_j}{1+|z|^2}\Big) +\frac{8\overline{\chi_{t,q}}(z) z_j}{(1+|z|^2)^3}
\Big\}\mathrm{d}z\\
&\quad+\int_{B_{\Gamma_{t,q}}(0)}\frac{O(1)(|\psi_t(z)||z_j|(e^{-\frac{\lambda_{t,q}}{2}}|z|+t|q|+t^2|\ln t|))}{(1+|z|^2)^3}\mathrm{d}z\\
&=\mathcal{K}_{1,t}+\mathcal{K}_{2,t}+\mathcal{K}_{3,t}.
\end{aligned}
\end{equation*}
For the right hand side of the above equation, using \eqref{bdd_of_psi_t} and \eqref{cutoff}, we get $\mathcal{K}_{1,t}=O(t)$ and $\mathcal{K}_{3,t}=O(t)$. In the end, we can show the term $\mathcal{K}_{2,t}=0$ from the equation $\Delta\Big(\frac{z_j}{1+|z|^2}\Big) +\frac{8z_j}{(1+|z|^2)^3}=0$. Together with Lemma \ref{bddchi2},  we prove   Lemma \ref{magofcc}.
\end{proof}

Next, we give a description on the asymptotic behavior of $\psi_t$ as $t\to0$.
\begin{lemma}
\label{limitfunctionofpsi_t}
There is a constant $d_0\in\mathbb{R}$ such that as ${t} \to0$,
\[\psi_t(z)\to d_0Y_0(z)\ \textrm{in}\  C^{0,\beta}_{\textrm{loc}}(\mathbb{R}^2 ).\]
\end{lemma}
\begin{proof}
From \eqref{bdd_of_psi_t}, we have  $\|\psi_t\|_{L^{\infty}(B_{\Gamma_{t,q}}(0))} \le O(1)$. We also note that
\[\|\Delta_z \psi_t(z)(1+|z|)^{1+\frac{\alpha}{2}}\|_{L^2(B_{\Gamma_{t,q}}(0))}=\| \Delta_z \overline{\phi}_t(z)(1+|z|)^{1+\frac{\alpha}{2}}\|_{L^2(B_{\Gamma_{t,q}}(0))}\le 1.\]
Using Sobolev embedding theorem, we can find a function $\psi_0$ such that $\psi_t\to\psi_0$ in $C^{0,\beta}_{\textrm{loc}}(\RN)$ and $\|\psi_0\|_{L^{\infty}(\RN)} \le O(1)$. Moreover, by Lemma \ref{magofcc} and
 $\|g_{t}\|_{Y_{\alpha,t,p}}=o(|\ln t|^{-1}),$
we can conclude that the limit of  equation \eqref{eq_for_psi} is
\begin{equation}
\label{limiteqinr2}
\begin{aligned}
\Delta\psi_0+\frac{8}{(1+|z|^2)^2}\psi_0=0\ \mbox{in}\ \RN.
\end{aligned}
\end{equation}
Since $\|\psi_0\|_{L^{\infty}(\RN)} \le O(1)$, we can apply \cite[Proposition 1]{bp} to get
\begin{equation}
\label{formofpsi_0}
\begin{aligned}
\psi_0(z) =\sum_{i=0}^2d_iY_i(z),
\end{aligned}
\end{equation}
where  $d_i$, $i=0, 1, 2$,  are some constants. By \eqref{haty}, we have
\[\overline{Z_{t,q,i}}(z)=Z_{t,q,i}(\Lambda_{t,q}^{-1}z+tq ) = \Lambda_{t,q}^2\Big(-\Delta_z (\overline{\chi_{t,q}}(z)Y_i(z))  +\frac{  8 \overline{\chi_{t,q}}^2(z)Y_i(z)}{(1+ |z|^2)^2}\Big).\]
By Lemma \ref{ztqi}-(i), we can derive $\phi_t-\frac{\int_Mh e^{U_{t,q} -G_t^{(2)} }\phi_t \mathrm{d}v_g}{\int_Mh e^{U_{t,q} -G_t^{(2)} }\mathrm{d}v_g}\in E_{\alpha,t,q,p}$ from $\phi_t\in E_{\alpha,t,q,p}$. Together with \eqref{haty}, we have
\begin{equation*}
\begin{aligned}
0&= \int_{B_{tR_0}(tq)}\Big(\phi_t(x)-\frac{\int_Mh e^{U_{t,q} -G_t^{(2)} }\phi_t \mathrm{d}v_g}{\int_Mh e^{U_{t,q} -G_t^{(2)} }\mathrm{d}v_g}\Big)Z_{t,q,i}(x)\mathrm{d}x\\
&= \int_{B_{\Gamma_{t,q}}(0)}\psi_t(z)\left(-\Delta(\overline{\chi_{t,q}}(z)Y_i(z))+\frac{8\overline{\chi_{t,q}}^2(z)Y_i(z)}{(1+|z|^2)^2}\right)\mathrm{d}z\\
&=\int_{B_{\Gamma_{t,q}}(0)\setminus B_{\frac{\Gamma_{t,q}}{2}}(0)}\psi_t(z)\left(-\frac{(\Delta\overline{\chi_{t,q}} )z_i}{(1+|z|^2)}  -2\nabla\overline{\chi_{t,q}} \cdot\nabla \Big(\frac{z_i}{(1+|z|^2)} \Big) \right) \mathrm{d}z\\
&\quad+\int_{B_{\Gamma_{t,q}}(0)}\psi_t(z) \frac{8\overline{\chi_{t,q}}(1+\overline{\chi_{t,q}} )Y_i(z)}{(1+|z|^2)^2}  \mathrm{d}z\\
&=\mathfrak{K}_{1,t,q}+\mathfrak{K}_{2,t,q}.
\end{aligned}
\end{equation*}
By \eqref{bdd_of_psi_t} and \eqref{cutoff}, we have $\mathfrak{K}_{1,t,q}=o(1)$. As a consequence, we can get $\mathfrak{K}_{2,t,q}=o(1)$ from the above equation. On the other hand, since $\psi_t\to\psi_0$ in $C^{0,\beta}_{\textrm{loc}}(\RN)$, we get that
$$0=\lim_{t\to0}\mathfrak{K}_{2,t,q}=\int_{\RN}\frac{ 16\psi_0(z) Y_i(z)}{(1+|z|^2)^2}  \mathrm{d}z,$$
which implies $d_i=0$, $i=1, 2$. Therefore, we get the conclusion.
\end{proof}

By $\int_{\RN}\frac{Y_0(z)}{(1+|z|^2)^2} \mathrm{d}z=\int_{\RN}\frac{(1-|z|^2)}{(1+|z|^2)^3} \mathrm{d}z=0$, \eqref{bdd_of_psi_t}, and Lemma \ref{limitfunctionofpsi_t}, we have
\begin{equation}
\begin{aligned}
\label{int_limit_of_psi_t}
\lim_{t\to0}\int_{B_{\Gamma_{t,q}}(0)}\frac{\psi_t(z)}{(1+|z|^2)^2} \mathrm{d}z=\int_{\RN}\frac{d_0Y_0(z)}{(1+|z|^2)^2} \mathrm{d}z=0.
\end{aligned}
\end{equation}
We will improve the estimation \eqref{int_limit_of_psi_t} in the following result. The proof makes use of the test function $\eta_1(z)=-Y_0(z)-1$.

\begin{lemma}
\label{improved_int_limit_of_psi_t}
(i) $\int_{B_{\Gamma_{t,q}}(0)}\frac{\psi_t(z)}{(1+|z|^2)^2} \mathrm{d}z=o(|\ln t|^{-1}).$

\noindent (ii) $ \int_{B_{\Gamma_{t,q}}(0)}(-\Delta\psi_t(z)) \mathrm{d}z=o(|\ln t|^{-1}).$
\end{lemma}
\begin{proof}Let $\eta_1(z)=-Y_0(z)-1=\frac{-2}{(1+|z|^2)}$.
Then $\eta_1$ satisfies
\begin{equation}
\begin{aligned}
\label{eq_test1}
\Delta\eta_1(z)+\frac{8\eta_1(z)}{(1+|z|^2)^2}=-\frac{8}{(1+|z|^2)^2}\ \mbox{in}\ \RN.
\end{aligned}
\end{equation}
Multiplying $\eta_1(z)\overline{\chi_{t,q}}(z)$ on both sides of \eqref{eq_for_psi} and using the integration by parts, we have
\begin{equation}
\label{improved_int_limit_of_psi_t1}
\begin{aligned}
0=&\int_{B_{\Gamma_{t,q}}(0)}\Bigg[\psi_t(z)\Bigg\{(\Delta\eta_1(z))\overline{\chi_{t,q}}(z)+2\nabla\eta_1(z)\nabla\overline{\chi_{t,q}}(z)
+\eta_1(z)(\Delta\overline{\chi_{t,q}}(z))\\
&+\Big(\frac{8+O(1)(e^{-\frac{\lambda_{t,q}}{2}}|z|+t|q|+t^2|\ln t|)}{(1+|z|^2)^2} \Big)\eta_1(z)\overline{\chi_{t,q}}(z)\Bigg\}\\
&-\Lambda_{t,q}^{-2}\Big(\overline{g_{t}}(z)+\sum_{i=1}^2c_{t,q,i}\overline{Z_{t,q,i}}(z)\Big)\eta_1(z)\overline{\chi_{t,q}}(z)\Bigg]\mathrm{d}z.
\end{aligned}
\end{equation}
Using H\"{o}lder inequality, we have
\begin{equation}
\label{improved_int_limit_of_psi_t01}
\begin{aligned}
&\int_{B_{\Gamma_{t,q}}(0)}\Lambda_{t,q}^{-2}\Big(|\overline{g_{t}}(z)|
+\sum_{i=1}^2|c_{t,q,i} \overline{Z_{t,q,i}}(z)|\Big)|\eta_1(z)\overline{\chi_{t,q}}(z)| \mathrm{d}z\\
&\le O(1)\Big( \|g_{t}\|_{Y_{\alpha,t,p}}+\sum_{i=1}^2|c_{t,q,i}|\|  Z_{t,q,i}\|_{Y_{\alpha,t,p}}\Big).
\end{aligned}
\end{equation}
By \eqref{eq_test1}, we see that
\begin{equation}
\label{improved_int_limit_of_psi_t2}
\begin{aligned}
\int_{B_{\Gamma_{t,q}}(0)}\psi_t(z)\Big\{(\Delta\eta_1(z))\overline{\chi_{t,q}}(z) +\frac{8\eta_1(z)\overline{\chi_{t,q}}(z)}{(1+|z|^2)^2} \Big\}\mathrm{d}z
=-\int_{B_{\Gamma_{t,q}}(0)}\frac{8\psi_t(z)\overline{\chi_{t,q}}(z)}{(1+|z|^2)^2}\mathrm{d}z.
\end{aligned}
\end{equation}
By \eqref{bdd_of_psi_t}, \eqref{cutoff} and $\eta_1(z)=O(1/(1+|z|^2))$ , we have
\begin{equation}
\label{improved_int_limit_of_psi_t3}
\begin{aligned}
&\int_{B_{\Gamma_{t,q}}(0)}\psi_t(z)\Big\{ 2\nabla\eta_1(z)\nabla\overline{\chi_{t,q}}(z)+\eta_1(z)(\Delta\overline{\chi_{t,q}}(z)) \Big\}\mathrm{d}z\\
&=\int_{B_{\Gamma_{t,q}}(0)\setminus B_{\frac{\Gamma_{t,q}}{2}}(0) }\psi_t(z)\Big\{2\nabla\eta_1(z)\nabla\overline{\chi_{t,q}}(z)
+\eta_1(z)(\Delta\overline{\chi_{t,q}}(z)) \Big\}\mathrm{d}z
=O(t^2),
\end{aligned}
\end{equation}
and
\begin{equation}
\label{improved_int_limit_of_psi_t4}
\begin{aligned}
& \int_{B_{\Gamma_{t,q}}(0)} \frac{ |\psi_t(z)|(e^{-\frac{\lambda_{t,q}}{2}}|z|+t|q|+t^2|\ln t|)}{(1+|z|^2)^2} \eta_1(z)\overline{\chi_{t,q}}(z)\mathrm{d}z
=O(t).
\end{aligned}
\end{equation}
From \eqref{improved_int_limit_of_psi_t1}-\eqref{improved_int_limit_of_psi_t4}, we have
\begin{equation}
\label{int_psi_with_weight}
\begin{aligned}
& \int_{B_{\Gamma_{t,q}}(0)}\frac{8\psi_t(z)\overline{\chi_{t,q}}(z)}{(1+|z|^2)^2}\mathrm{d}z
=O(1)(\|g_{t}\|_{Y_{\alpha,t,p}}+\sum_{i=1}^2|c_{t,q,i}|  \|Z_{t,q,i}\|_{Y_{\alpha,t,p}}+t).
\end{aligned}
\end{equation}
Together with $\|g_{t}\|_{Y_{\alpha,t,p}}=o(|\ln t|^{-1})$  and Lemma \ref{magofcc},  we prove the Lemma \ref{improved_int_limit_of_psi_t}-(i).

By \eqref{eq_for_psi}, we see that
\begin{equation}
\label{int_eq_for_psi}
\begin{aligned}
&\int_{B_{\Gamma_{t,q}}(0)}(-\Delta\psi_t(z))\mathrm{d}z\\
&=\int_{B_{\Gamma_{t,q}}(0)}\frac{ 8 \psi_t(z)}{(1+|z|^2)^2}\mathrm{d}z+\int_{B_{\Gamma_{t,q}}(0)}\frac{  O(1)(|\psi_t(z)|e^{-\frac{\lambda_{t,q}}{2}}|z|+t|q|+t^2|\ln t|)}{(1+|z|^2)^2}\mathrm{d}z\\
&\quad+O(1)(\|g_{t}\|_{Y_{\alpha,t,p}}).
\end{aligned}
\end{equation}
By \eqref{int_psi_with_weight} and $\|g_{t}\|_{Y_{\alpha,t,p}}=o(|\ln t|^{-1})$, we get the Lemma \ref{improved_int_limit_of_psi_t}-(ii).
\end{proof}

By Lemma \ref{improved_int_limit_of_psi_t}, we get the following result.
\begin{lemma}
\label{int_limit_outside}
$\frac{\int_Mh e^{U_{t,q} -G_t^{(2)} }\phi_t \mathrm{d}v_g}{\int_Mh e^{U_{t,q} -G_t^{(2)} }\mathrm{d}v_g}= \frac{\int_{M\setminus B_{t{R_0}}(tq)}h e^{w}\phi_t \mathrm{d}v_g}{\int_Mh e^{w }\mathrm{d}v_g}+o(|\ln t|^{-1}).$
\end{lemma}
\begin{proof} In order to estimate $\frac{\int_Mh e^{U_{t,q} -G_t^{(2)} }\phi_t \mathrm{d}v_g}{\int_Mh e^{U_{t,q} -G_t^{(2)} }\mathrm{d}v_g}$, we divide $M$ into $M\setminus B_{t{R_0}}(tq)$ and $B_{t{R_0}}(tq)$. Using Lemma \ref{basic_est_for_app}, we have
\begin{equation}
\label{global_mass2}
\begin{aligned}
&\frac{\int_{M\setminus B_{t{R_0}}(tq)}h e^{U_{t,q} -G_t^{(2)} }\phi_t \mathrm{d}v_g}{\int_Mh e^{U_{t,q} -G_t^{(2)} }\mathrm{d}v_g}\\
&=\left(\frac{\rho-8\pi }{\rho }\right)\frac{ \int_{M\setminus B_{t{R_0}}(tq)}h(x)e^{w(x)}\phi_t \mathrm{d}v_g}{ \int_Mhe^w\mathrm{d}v_g}\\
&\quad+\int_{M\setminus B_{t{R_0}}(tq)}O(|\phi_t(x)|)\Big(1_{B_{\mathbf{r}_0}(0)}(x)\left(\frac{t|q||x-tq|+t^2}{|x-tq|^2}\right)+t^2|\ln t|+t|q|\Big) \mathrm{d}v_g,
\end{aligned}
\end{equation}
and
\begin{equation}
\label{global_massin}
\begin{aligned}
&\frac{\int_{B_{t{R_0}}(tq)}h e^{U_{t,q} -G_t^{(2)} }\phi_t \mathrm{d}v_g}{\int_Mh e^{U_{t,q} -G_t^{(2)} }\mathrm{d}v_g}
\\&=\int_{B_{t{R_0}}(tq)}\frac{C_{t,q}^2\Lambda_{t,q}^2H_{t,q}\left(t^{-1}x\right)(1+O(t|q|+t^2|\ln t|))\phi_t(x)}{(1+\Lambda_{t,q}^2|x-tq|^2)^2}\mathrm{d}x.
\end{aligned}
\end{equation}
From the changing of variable $x\to\Lambda_{t,q}^{-1}z+tq$, we get from \eqref{global_massin} that
\begin{equation}
\label{global_massin2}
\begin{aligned}
\frac{\int_{B_{t{R_0}}(tq)}h e^{U_{t,q} -G_t^{(2)}}\phi_t \mathrm{d}v_g}{\int_Mh e^{U_{t,q} -G_t^{(2)}}\mathrm{d}v_g}
&=\int_{B_{\Gamma_{t,q}}(0)}\frac{8\overline{\phi_t}(z)\frac{H_{t,q}(\Lambda_{t,q}^{-1}t^{-1}z+q)}{H_{t,q}(q)}
(1+O(t|q|+t^2|\ln t|))}{\rho(1+|z|^2)^2}\mathrm{d}z\\
&=\int_{B_{\Gamma_{t,q}}(0)} \frac{8\overline{\phi_t}(z)(1+O(e^{-\frac{\lambda_{t,q}}{2}}|z|+t|q|+t^2|\ln t|))}{\rho(1+|z|^2)^2}\mathrm{d}z\\
&=\int_{B_{\Gamma_{t,q}}(0)} \frac{8\overline{\phi_t}(z) }{\rho(1+|z|^2)^2}  \mathrm{d}z +O\left(t\|\phi_t\|_{L^{\infty}(M)}\right).
\end{aligned}
\end{equation}
Moreover, from Lemma \ref{improved_int_limit_of_psi_t} and the definition of $\psi_t$, we have
\begin{equation}
\label{global_mass1}
\begin{aligned}
o(|\ln t|^{-1})=&\int_{B_{\Gamma_{t,q}}(0)}\frac{8{\psi_t}(z)}{(1+|z|^2)^2} \mathrm{d}z
=\int_{B_{\Gamma_{t,q}}(0)}\frac{8\left(\overline{\phi_t}(z)-\frac{\int_Mh e^{U_{t,q} -G_t^{(2)} }\phi_t \mathrm{d}v_g}{\int_Mh e^{U_{t,q} -G_t^{(2)} }\mathrm{d}v_g}\right)}{(1+|z|^2)^2}\mathrm{d}z \\
=&\int_{B_{\Gamma_{t,q}}(0)}\frac{8\overline{\phi_t}(z)}{(1+|z|^2)^2}\mathrm{d}z-\frac{\int_Mh e^{U_{t,q} -G_t^{(2)} }\phi_t \mathrm{d}v_g}{\int_Mh e^{U_{t,q} -G_t^{(2)} }\mathrm{d}v_g}(8\pi+O(t^2)).
\end{aligned}
\end{equation}
Using \eqref{global_mass2}-\eqref{global_mass1}, we finally get that
\begin{equation*}
\begin{aligned}
&\frac{\int_Mh e^{U_{t,q} -G_t^{(2)} }\phi_t \mathrm{d}v_g}{\int_Mh e^{U_{t,q} -G_t^{(2)} }\mathrm{d}v_g}
=\Big(\int_{M\setminus B_{t{R_0}}(tq)}+\int_{B_{t{R_0}}(tq)}\Big)\frac{h e^{U_{t,q} -G_t^{(2)} }\phi_t}{\int_Mh e^{U_{t,q} -G_t^{(2)} }\mathrm{d}v_g}\mathrm{d}v_g\\
&=\left(1-\frac{8\pi }{\rho }\right)\frac{ \int_{M\setminus B_{t{R_0}}(tq)}h(x)e^{w(x)}\phi_t \mathrm{d}v_g}{ \int_Mhe^w\mathrm{d}v_g}
+\frac{8\pi}{\rho}\frac{\int_Mh e^{U_{t,q} -G_t^{(2)} }\phi_t \mathrm{d}v_g}{\int_Mh e^{U_{t,q} -G_t^{(2)} }\mathrm{d}v_g} +o(|\ln t|^{-1}),
\end{aligned}
\end{equation*}
which proves Lemma \ref{int_limit_outside}.
\end{proof}
For any function $g$ satisfying $g(z)(1+|z|)^{1+\frac{\alpha}{2}}\in L^2(\RN)$, we recall the following inequality, which will be useful for our later arguments.
\begin{lemma}
\cite{cfl}
\label{estfory}
There is a constant $c >0$, independent of $x\in\RN\setminus B_2(0)$ and $g$, such that
\[\left|\intr (\ln|x-z|-\ln|x|)g(z)\mathrm{d}z\right|\le c  |x|^{-\frac{\alpha}{2}}(\ln|x|+1)\|g(z)(1+|z|)^{1+\frac{\alpha}{2}}\|_{L^2(\RN)}.\]
\end{lemma}
By H\"{o}lder inequality and Lemma \ref{magofcc}, we get the following estimation for $\phi_t$ satisfying \eqref{c}.
\begin{lemma}
\label{phitalphat}
$\int_M|\Delta \phi_t|\mathrm{d}v_g \le O(1)(\|\phi_t\|_{X_{\alpha,t,p}})
\le O(1)(\|\phi_t\|_{L^{\infty}(M)})+o(|\ln t|^{-1}).$
\end{lemma}
\begin{proof}
By H\"{o}lder inequality and the change of variables $x=\Lambda_{t,q}^{-1}z+tq\in B_{tR_0}(tq)$, we have
\begin{equation}
\label{xal1}
\begin{aligned}
\int_M|\Delta \phi_t|\mathrm{d}v_g
&\le O(1)\Big(  \| \Delta \phi_t \|_{L^p(M\setminus B_{tR_0}(tq))}+\int_{B_{\Gamma_{t,q}}(0)}|\Delta_z \overline{\phi_t}(z)|\mathrm{d}z\Big)\\
&\le  O(1)\Big(\| \Delta \phi_t \|_{L^p(M\setminus B_{tR_0}(tq))}+\|\Delta_z \overline{\phi_t}(z)(1+|z|)^{1+\frac{\alpha}{2}}\|_{L^2(B_{\Gamma_{t,q}}(0))}\Big)\\
&\le O(1)( \|\phi_t\|_{X_{\alpha,t,p}}). \end{aligned}
\end{equation}
On the other hand, we see that
\begin{equation}
\label{xal2}
\begin{aligned}
\|\overline{\phi}_t(z)\rho(z)\|_{L^2(B_{\Gamma_{t,q}}(0))}+\|\phi_t\|_{L^p(M\setminus B_{\frac{t{R_0}}{2}}(tq))}
\le O(1)( \|\phi_t\|_{L^{\infty}}),
\end{aligned}
\end{equation}
where we used $\|\rho\|_{ L^2(B_{\Gamma_{t,q}}(0))}=O(1).$
Combined with equation \eqref{c} and \eqref{eq_for_psi}, we get from \eqref{bdd_of_psi_t} and Lemma \ref{basic_est_for_app}-(iii),
\begin{equation}
\label{xal3}
\begin{aligned}
&\|\Delta\overline{\phi}_t(z)(1+|z|)^{1+\frac{\alpha}{2}}\|_{L^2(B_{\Gamma_{t,q}}(0))}
+\|\Delta \phi_t\|_{L^p(M\setminus B_{\frac{t{R_0}}{2}}(tq))}\\
&\le O(1)\Bigg(\left\|\left(\frac{\psi_t(z)}{(1+|z|^2)^2}\right)(1+|z|)^{1+\frac{\alpha}{2}}\right\|_{L^2(B_{\Gamma_{t,q}}(0))}\\
&\quad+\left\|\frac{h(x)e^{U_{t,q}(x)-G_t^{(2)}(x)}}{\int_Mh e^{U_{t,q} -G_t^{(2)} }\mathrm{d}v_g}\Big(\phi_t(x)-\frac{\int_Mh e^{U_{t,q} -G_t^{(2)} }\phi_t \mathrm{d}v_g}{\int_Mh e^{U_{t,q} -G_t^{(2)} }\mathrm{d}v_g}\Big)\right\|_{L^p(M\setminus B_{\frac{t{R_0}}{2}}(tq))}\\
&\quad+\|g_{t}\|_{Y_{\alpha,t,p}}+\sum_{i=1}^2|c_{t,q,i}|\|Z_{t,q,i}\|_{Y_{\alpha,t,p}}\Bigg)\\
&\le O(1)\Big( \|\phi_t\|_{L^{\infty}(M)}+\|g_{t}\|_{Y_{\alpha,t,p}}+\sum_{i=1}^2|c_{t,q,i}|\|Z_{t,q,i}\|_{Y_{\alpha,t,p}}\Big).
\end{aligned}
\end{equation}
By \eqref{xal2}- \eqref{xal3}, we get $$\|\phi_t\|_{X_{\alpha,t,p}}\le O(1)\Big(\|\phi_t\|_{L^{\infty}(M)}+\|g_{t}\|_{Y_{\alpha,t,p}}
+\sum_{i=1}^2|c_{t,q,i}|\|Z_{t,q,i}\|_{Y_{\alpha,t,p}}\Big).$$
Together with $\|g_{t}\|_{Y_{\alpha,t,p}}=o(|\ln t|^{-1})$ and Lemma \ref{magofcc}, we get Lemma \ref{phitalphat}. \end{proof}
By the Green representation formula and Lemma \ref{estfory}, we  compare the differences for the value of $\phi_t$ in different regions.
\begin{lemma}
\label{diff_of_phi_t}
(i) If  $|x-tq|\le |x'-tq|$ and $x, x'\in B_{\mathbf{r}_0}(0)\setminus B_{\frac{t{R_0}}{2}}(tq)$, then
\begin{equation*}
\begin{aligned}
&|\phi_t(x)-\phi_t(x')|\\&\le
O(1)\Big(\Big|\ln\frac{ |x'-tq|}{| x-tq|}\int_{B_{\Gamma_{t,q}}(0)}\Delta\psi_t(z)\mathrm{d}z\Big|
+|x'-tq|^{\frac{2(p-1)}{p}}|\ln|x'-tq||\Big)\\
&+O(\lambda_{t,q}e^{-\frac{\alpha\lambda_{t,q}}{4}}).
\end{aligned}
\end{equation*}



(ii) For $x\in  B_{\frac{t{R_0}}{2}}(tq)\setminus B_{t^2R_0} (tq)$,

\begin{equation*}
\begin{aligned}
&|\phi_t(x)-\phi_t(tq)|\\&\le O(1)\Big(
\Big|\int_{B_{\Gamma_{t,q}}(0)}\ln|z|\Delta\psi_t(z)\mathrm{d}z\Big|
+\Big|\ln\Big(\Lambda_{t,q}|x-tq|\Big)\int_{B_{\Gamma_{t,q}}(0)}\Delta\psi_t(z)\mathrm{d}z\Big|\Big)\\
&+O(1)\Big(t^{\frac{2(p-1)}{p}}|\ln t|+\Big(\frac{|x-tq|}{t^2}\Big)^{-\frac{\alpha}{2}}\ln\Big(\frac{|x-tq|}{t^2}\Big)\Big).
\end{aligned}
\end{equation*}
\end{lemma}

\begin{proof}
\textit{Step 1}.
By the Green representation formula,  we have for any $x, x'\in B_{\mathbf{r}_0}(0)$,
\begin{equation}
\begin{aligned}
\label{after_green_representation1}
\phi_t(x)-\phi_t(x')=\int_M(G(x',\zeta)-G(x,\zeta)) \Delta\phi_t(\zeta) \mathrm{d}\zeta.
\end{aligned}
\end{equation}
By Lemma \ref{phitalphat}, we see that
\begin{equation}
\begin{aligned}
\label{after_green_representation2}
&\int_{M\setminus B_{2\mathbf{r}_0}(0)} (G(x,\zeta)-G(x',\zeta))\Delta\phi_t(\zeta)\mathrm{d}\zeta
+\int_{ B_{2\mathbf{r}_0}(0)} (R(x,\zeta)-R(x',\zeta)) \Delta\phi_t(\zeta) \mathrm{d}\zeta\\
&\le O(1)(|x-x'|\|\phi_t\|_{X_{\alpha,t,p}})\le O(1)(|x-x'|).
\end{aligned}
\end{equation}


\noindent \textit{Step 2}. Suppose that $|x-tq|\le |x'-tq|$  and $x, x'\in B_{\mathbf{r}_0}(0)\setminus B_{\frac{t{R_0}}{2}}(tq)$. By H\"{o}lder inequality, we see that for some  $\theta\in(0,1)$,
\begin{equation}
\begin{aligned}
\label{after_green_representation3}
&\Big|\int_{  B_{2\mathbf{r}_0}(0)\setminus B_{ t{R_0} }(tq)}\ln\left(\frac{|x'-\zeta|}{|x-\zeta|}\right) \Delta\phi_t(\zeta) \mathrm{d}\zeta\Big|\\
&\le\int_{  B_{2\mathbf{r}_0}(0)\setminus B_{ 2|x'-tq|} (tq)}\frac{ |x'-x| }{\theta |x'-\zeta|+(1-\theta)|x-\zeta|} |\Delta\phi_t(\zeta)| \mathrm{d}\zeta\\
&\quad+\int_{  B_{ 2|x'-tq|} (tq)\setminus B_{ t{R_0} }(tq)}(|\ln|x'-\zeta||+|\ln|x-\zeta||)|\Delta\phi_t(\zeta)| \mathrm{d}\zeta\\
&\le O(1)\Big( |x'-tq|^{\frac{2(p-1)}{p}}|\ln|x'-tq||\|\Delta\phi_t \|_{L^p( B_{2\mathbf{r}_0}(0)\setminus B_{t{R_0}}(tq))}\Big)\\
&\le O(1)(|x'-tq|^{\frac{2(p-1)}{p}}|\ln|x'-tq||),
\end{aligned}
\end{equation}
where we used $|x-x'|\leq 2|x'-tq|$. By the change of variables, we also see that
\begin{equation}
\begin{aligned}
\label{after_green_representation4}
&\int_{  B_{ t{R_0} }(tq)}\ln\left(\frac{|x'-\zeta|}{|x-\zeta|}\right) \Delta\phi_t(\zeta) \mathrm{d}\zeta\\
&=\int_{B_{\Gamma_{t,q}}(0)}\ln\left(\frac{|x'-tq-\Lambda_{t,q}^{-1}z|}{|x-tq-\Lambda_{t,q}^{-1}z|}\right)\Delta\psi_t(z) \mathrm{d}z\\
&=\int_{B_{\Gamma_{t,q}}(0)}\ln\left(\frac{\left|\Lambda_{t,q}(x'-tq)-z\right|}{\left|\Lambda_{t,q}(x-tq)-z\right|}\right) \Delta\psi_t(z)\mathrm{d}z.
\end{aligned}
\end{equation}
Let
$$\mathfrak{Z}_{t,q}=\ln\left(\frac{\left|\Lambda_{t,q}(x'-tq)\right|}{\left|\Lambda_{t,q}(x-tq)\right|}\right) \int_{B_{\Gamma_{t,q}}(0)}\Delta\psi_t  \mathrm{d}z=\ln \left(\frac{| x'-tq |}{ | x-tq |}\right) \int_{B_{\Gamma_{t,q}}(0)}\Delta\psi_t \mathrm{d}z.$$
By  adding  and  substituting
the same constant $\mathfrak{Z}_{t,q}$ in the last line of \eqref{after_green_representation4}, we get
\begin{equation}
\begin{aligned}
\label{after_green_representation45}
&\int_{B_{ t{R_0} }(tq)}\ln\left(\frac{|x'-\zeta|}{|x-\zeta|}\right) \Delta\phi_t(\zeta) \mathrm{d}\zeta\\
&=\mathfrak{Z}_{t,q}\\&+\int_{B_{\Gamma_{t,q}}(0)}
\left\{\ln\left(\frac{\left|\Lambda_{t,q}(x'-tq)-z\right|}{\left|\Lambda_{t,q}(x'-tq)\right|}\right)
-\ln\left(\frac{\left|\Lambda_{t,q}(x-tq)-z\right|}{\left|\Lambda_{t,q}(x-tq)\right|}\right)\right\}\Delta\psi_t(z) \mathrm{d}z.
\end{aligned}
\end{equation}
Since
$$\frac12\Lambda_{t,q}R_0t\le \left|\Lambda_{t,q}(x-tq)\right|\le \left|\Lambda_{t,q}(x'-tq)\right|,$$
by applying Lemma \ref{estfory}, we see that
\begin{equation}
\begin{aligned}
\label{after_green_representation5}
&\int_{B_{ t{R_0} }(tq)}\ln\left(\frac{|x'-\zeta|}{|x-\zeta|}\right) \Delta\phi_t(\zeta) \mathrm{d}\zeta
=\mathfrak{Z}_{t,q}+O(\lambda_{t,q}e^{-\frac{\alpha\lambda_{t,q}}{4}}).
\end{aligned}
\end{equation}
By \eqref{after_green_representation1}-\eqref{after_green_representation5}, we obtain Lemma \ref{diff_of_phi_t}-(i).


\medskip\noindent

\noindent\textit{Step 3}. Suppose that $x\in  B_{\frac{t{R_0}}{2}}(tq)\setminus B_{t^2R_0} (tq)$. By H\"{o}lder inequality, we see that for some  $\theta\in(0,1)$,
\begin{equation}
\begin{aligned}
\label{after_green_representation6}&
\Big|\int_{  B_{2\mathbf{r}_0}(0)\setminus B_{ t{R_0} }(tq)}\ln\left(\frac{|tq-\zeta|}{|x-\zeta|}\right) \Delta\phi_t(\zeta) \mathrm{d}\zeta\Big|\\
&\le\int_{  B_{2\mathbf{r}_0}(0)\setminus B_{ t{R_0} }(tq)}\frac{ |tq-x| }{\theta |tq-\zeta|+(1-\theta)|x-\zeta|} |\Delta\phi_t(\zeta)| \mathrm{d}\zeta\\
&\le O(1)\Big( |x-tq|t^{\frac{p-2}{p}}|\ln t|\|\Delta\phi_t \|_{L^p( B_{2\mathbf{r}_0}(0)\setminus B_{ t{R_0} }(tq))}\Big)
\le O(1)( t^{\frac{2(p-1)}{p}}|\ln t|).
\end{aligned}
\end{equation}
By the change of variables, we also see that
\begin{equation}
\begin{aligned}
\label{after_green_representation7}&
\int_{  B_{ t{R_0} }(tq)}\ln\Big(\frac{|tq-\zeta|}{|x-\zeta|}\Big) \Delta\phi_t(\zeta) \mathrm{d}\zeta\\
&=\int_{B_{\Gamma_{t,q}}(0)}\ln\left(\frac{|\Lambda_{t,q}^{-1}z|}{|x-tq-\Lambda_{t,q}^{-1}z|}\right) \Delta\psi_t(z) \mathrm{d}z\\
&=\int_{B_{\Gamma_{t,q}}(0)}\ln\left(\frac{ |z |}{\left|\Lambda_{t,q}(x-tq)- z\right|}\right) \Delta\psi_t(z) \mathrm{d}z.
\end{aligned}
\end{equation}
Let
$$\Upsilon_{t,q}=\left(\ln\left|\Lambda_{t,q}(x-tq)\right|\int_{B_{\Gamma_{t,q}}(0)}\Delta\psi_t \mathrm{d}z\right).$$
By adding and substituting the same constant $\Upsilon_{t,q}$  in the last line of \eqref{after_green_representation7}, we get
\begin{equation}
\begin{aligned}
\label{after_green_representation71}
&\int_{  B_{ t{R_0} }(tq)}\ln\left(\frac{|tq-\zeta|}{|x-\zeta|}\right) \Delta\phi_t(\zeta) \mathrm{d}\zeta\\
&=\int_{B_{\Gamma_{t,q}}(0)}\left\{\ln|z|
-\ln\left(\frac{\left|\Lambda_{t,q}(x-tq)-z\right|}{\left|\Lambda_{t,q}(x-tq)\right|}\right)\right\}
\Delta\psi_t(z) \mathrm{d}z-\Upsilon_{t,q}.
\end{aligned}
\end{equation}
We note that $\left|\Lambda_{t,q}(x-tq)\right|\ge \Lambda_{t,q}R_0t^2\ge c_0>0$ for some constant $c_0>0$  independent of $t>0$.
By applying  Lemma \ref{estfory}, we see that
\begin{equation}
\begin{aligned}
\label{after_green_representation8}
&\int_{  B_{ t{R_0} }(tq)}\ln\left(\frac{|tq-\zeta|}{|x-\zeta|}\right) \Delta\phi_t(\zeta) \mathrm{d}\zeta\\
&=\int_{B_{\Gamma_{t,q}}(0)} \ln | z|\Delta\psi_t(z) \mathrm{d}z
-\ln\left|\Lambda_{t,q}(x-tq)\right|\int_{B_{\Gamma_{t,q}}(0)}\Delta\psi_t(z)\mathrm{d}z\\ 
&\quad+O(1)\Big(\frac{|x-tq|}{t^2}\Big)^{-\frac{\alpha}{2}}\left|\ln \frac{|x-tq|}{t^2}\right|.
\end{aligned}\end{equation}By \eqref{after_green_representation1}-\eqref{after_green_representation2} and \eqref{after_green_representation6}-\eqref{after_green_representation8}, we obtain Lemma \ref{diff_of_phi_t}-(ii).
 \end{proof}

\begin{lemma}
\label{vanish_outside}
(i) $\phi_t\to 0\ \textrm{in}\  C^{0,\beta}_{\textrm{loc}}(M\setminus\{0\}) $ as $t\to0$.

(ii) $\|\phi_t\|_{L^{\infty}(M\setminus B_{\frac{t{R_0}}{2}}(tq))}=o(1)$.

(iii) $\frac{\int_Mh e^{U_{t,q} -G_t^{(2)} }\phi_t \mathrm{d}v_g}{\int_Mh e^{U_{t,q} -G_t^{(2)} }\mathrm{d}v_g}=o(1)$.
\end{lemma}

\begin{proof}
(i) Since $\|\phi_t\|_{L^{\infty}(M)}+\|\Delta\phi_t\|_{L^p(M\setminus B_{t{R_0}}(tq))}\le 1$, by the Sobolev imbedding theorem, $\phi_t\to\phi_0$ in $C^{0,\beta}_{\textrm{loc}}(M\setminus\{0\})$. In addition, we get $\|\phi_0\|_{L^\infty(M)}\le 1$ and $\int_M\phi_0\mathrm{d}v_g=0$ from $\|\phi_t\|_{L^\infty(M)}\le 1$ and $\int_M\phi_t\mathrm{d}v_g=0$.
By Lemma \ref{basic_est_for_app}, Lemma  \ref{int_limit_outside}, Lemma \ref{magofcc},  and $\|g_{t}\|_{L^p(M\setminus B_{t{R_0}}(tq))}=o(|\ln t|^{-1})$, the equation \eqref{c} implies
\begin{equation}
\label{eq_for_phi0}
\Delta \phi_0+(\rho-8\pi)\frac{h(x)e^{w(x)}}{\int_M he^w\mathrm{d}v_g}\left(\phi_0-\frac{\int_Mhe^{w}\phi_0\mathrm{d}v_g}{\int_M he^w\mathrm{d}v_g}\right)=0
\ \textrm{in}\ M\setminus\{0\}.
\end{equation}
Since $\|\phi_0\|_{L^{\infty}(M)}\le1$,   $\phi_0$ is smooth near $0$, then we can extend the equation \eqref{eq_for_phi0} to $M$.
By the non-degeneracy assumption for \eqref{eq_for_phi0} and $\int_M\phi_0\mathrm{d}v_g=0$, we get $\phi_0\equiv0$ in $M$ and it proves (i).

\medskip

\noindent (ii) We prove Lemma \ref{vanish_outside}-(ii) by contradiction. Suppose for $t>0$ small,
$$\|\phi_t\|_{L^{\infty}(M\setminus B_{\frac{t{R_0}}{2}}(tq))}=|\phi_t(x_t)|\ge c_0>0,$$
where $c_0>0$ is a constant  independent of $t>0$. By Lemma \ref{vanish_outside}-(i), we have $\lim_{t\to0}|x_t|=0$.
Using Lemma \ref{improved_int_limit_of_psi_t}-(ii) and Lemma \ref{diff_of_phi_t}, we get for any  $d\in (0,\mathbf{r}_0)$,
\begin{equation}
\begin{aligned}
\label{outside_smallarea}
|\phi_t(x_t)-\phi_t(tq+d\vec{e})| \le c d^{\frac{2(p-1)}{p}}|\ln d|  +o(1),
\end{aligned}
\end{equation}where $c>0$ is a constant, independent of $t>0$ and $d\in(0,\mathbf{r}_0)$,
We choose $d>0$ such that $cd^{\frac{2(p-1)}{p}}|\ln d| \le \frac{c_0}{4}$ and fix such $d$. Combined with  Lemma \ref{vanish_outside}-(i), we get   $\lim_{t\to0}|\phi_t(x_t)|\le \frac{c_0}{2}$, which contradicts to $|\phi_t(x_t)|\ge c_0>0$. So  Lemma \ref{vanish_outside}-(ii) holds.

\medskip

\noindent (iii) By Lemma \ref{int_limit_outside}, we have
\begin{equation}
\begin{aligned}
\frac{\int_Mh e^{U_{t,q} -G_t^{(2)} }\phi_t \mathrm{d}v_g}{\int_Mh e^{U_{t,q} -G_t^{(2)} }\mathrm{d}v_g}= \frac{\int_{M\setminus B_{t{R_0}}(tq)}h e^{w}\phi_t \mathrm{d}v_g}{\int_Mh e^{w }\mathrm{d}v_g}+o(1).
\end{aligned}
\end{equation}
Then  Lemma \ref{vanish_outside}-(iii) follows from  Lemma \ref{vanish_outside}-(ii).
\end{proof}

In order to get the rid of the factor $d_0$ from dilations (see Lemma \ref{limitfunctionofpsi_t}), we need to introduce the function
$$\eta_2(z)=\frac{4}{3}\ln(1+|z|^2)\left(\frac{1-|z|^2}{1+|z|^2}\right)+\frac{8}{3(1+|z|^2)},$$
which satisfies
\begin{equation}
\begin{aligned}
\label{eq_test2}
\Delta\eta_2(z)+\frac{8\eta_2(z)}{(1+|z|^2)^2}=\frac{16(1-|z|^2)}{(1+|z|^2)^3}=\frac{16Y_0(z)}{(1+|z|^2)^2}\quad \mbox{in}\ \RN.
\end{aligned}
\end{equation}
The function $\eta_2$ was firstly introduced by Esposito,  Grossi, and  Pistoia in \cite{egp}, and then also used in \cite{emp,f,ly1} later.

\begin{lemma}
\label{vanish_insidew}
(i) $ \psi_t(z) \to 0$ and $\overline{\phi_t}(z) \to 0$ in  $C^{0,\beta}_{\textrm{loc}}(\mathbb{R}^2 )$ as $t\to0$.

(ii) For any fixed constant $R>0$, $\|\phi_t\|_{L^{\infty}(B_{t^2R}(tq))}=o(1)$.

(iii)  $\|\phi_t\|_{L^{\infty}(M)}=o(1)$.
\end{lemma}

\begin{proof}
(i). Multiplying $\eta_2(z)\overline{\chi_{t,q}}(z)$ on \eqref{eq_for_psi} and using the integration by parts, we see that
\begin{equation}
\label{last1}
\begin{aligned}
&\int_{B_{\Gamma_{t,q}}(0)}\psi_t(z)\Big\{\Delta(\eta_2(z)\overline{\chi_{t,q}}(z))+\frac{8(1+O(t|z|+t|q|+t^2|\ln t|))}{(1+|z|^2)^2} \eta_2(z)\overline{\chi_{t,q}}(z)\Big\}\mathrm{d}z\\
&=\int_{B_{\Gamma_{t,q}}(0)}\Lambda_{t,q}^{-2}\overline{g_{t}}(z) \eta_2(z)\overline{\chi_{t,q}}(z)\mathrm{d}z.
\end{aligned}
\end{equation}
By H\"{o}lder inequality, we  see that
\begin{equation}
\label{last2}
\int_{B_{\Gamma_{t,q}}(0)}
\Lambda_{t,q}^{-2}|\overline{g_{t}}(z) \eta_2(z)\overline{\chi_{t,q}}(z)|  \mathrm{d}z\le O(1)(\|g_{t}\|_{Y_{\alpha,t,p}})=o(|\ln t|^{-1}).
\end{equation}
By the equation \eqref{eq_test2}, we see that
\begin{equation}
\label{last3}
\begin{aligned}
&\int_{B_{\Gamma_{t,q}}(0)}\psi_t(z)\Big\{(\Delta
\eta_2(z))\overline{\chi_{t,q}}(z) +\frac{8}{(1+|z|^2)^2} \eta_2(z)\overline{\chi_{t,q}}(z)\Big\}\mathrm{d}z\\
&= \int_{B_{\Gamma_{t,q}}(0)}\frac{16\psi_t(z)\overline{\chi_{t,q}}(z) (1-|z|^2)}{(1+|z|^2)^3}\mathrm{d}z.
\end{aligned}
\end{equation}
By \eqref{bdd_of_psi_t}, we have $\|\psi_t\|_{L^{\infty}(B_{\Gamma_{t,q}}(0))}\le O(1)$, and thus
\begin{equation}
\label{last4}
\begin{aligned}
&\int_{B_{\Gamma_{t,q}}(0)} \frac{|\psi_t(z)|(t|z|+t|q|+t^2|\ln t|)}{(1+|z|^2)^2} \eta_2(z)\overline{\chi_{t,q}}(z)\mathrm{d}z=O(t).
\end{aligned}
\end{equation}
From \eqref{last1}-\eqref{last4} and $\textrm{supp}\left(\nabla\overline{\chi_{t,q}}\right)\subseteq B_{\Gamma_{t,q}}(0)\setminus B_{\frac{\Gamma_{t,q}}{2}}(0)$, we get that
\begin{equation}
\label{int_psi_with_kernal}
\begin{aligned}
& -\int_{B_{\Gamma_{t,q}}(0)}\frac{16\psi_t(z)\overline{\chi_{t,q}}(z)(1-|z|^2)}{(1+|z|^2)^3}\mathrm{d}z\\
&= \int_{B_{\Gamma_{t,q}}(0)\setminus B_{\frac{\Gamma_{t,q}}{2}}(0) }\psi_t(z)\left(2\nabla\eta_2(z)\cdot\nabla\overline{\chi_{t,q}}(z)+\eta_2(z)(\Delta\overline{\chi_{t,q}}(z))\right)\mathrm{d}z +o(1).
\end{aligned}
\end{equation}
Let $$\mathfrak{L}_{t,q}=\psi_t(\Lambda_{t,q}tR_0\vec{e})\int_{B_{\Gamma_{t,q}}(0)}\left(2\nabla\eta_2 \cdot\nabla\overline{\chi_{t,q}} +\eta_2  \Delta\overline{\chi_{t,q}} \right)\mathrm{d}z.$$
By adding and  substituting the same constant $\mathfrak{L}_{t,q}$ in the second line of \eqref{int_psi_with_kernal}, we see that
\begin{equation}
\label{int_psi_with_kernal2}
\begin{aligned}
&\int_{B_{\Gamma_{t,q}}(0)}\psi_t(z)\Big(2\nabla\eta_2(z)\cdot\nabla\overline{\chi_{t,q}}(z)+\eta_2(z)(\Delta\overline{\chi_{t,q}}(z))\Big)\mathrm{d}z\\
&= \int_{B_{\Gamma_{t,q}}(0)}\left(\psi_t(z)-\psi_t(\Lambda_{t,q}tR_0\vec{e})\right)
\Big(2\nabla\eta_2\cdot\nabla\overline{\chi_{t,q}}+\eta_2(\Delta\overline{\chi_{t,q}})\Big)\mathrm{d}z+\mathfrak{L}_{t,q}\\
&= \int_{B_{\Gamma_{t,q}}(0)}\left(\phi_t(\Lambda_{t,q}^{-1}z+tq)-\phi_t(t{R_0}\vec{e}+tq)\right)2\nabla\eta_2\cdot\nabla\overline{\chi_{t,q}} \mathrm{d}z\\
&\quad+\int_{B_{\Gamma_{t,q}}(0)}\left(\phi_t(\Lambda_{t,q}^{-1}z+tq)-\phi_t(t{R_0}\vec{e}+tq)\right) \eta_2(z)\Delta\overline{\chi_{t,q}} \mathrm{d}z\\
&\quad+\Big(\phi_t (t{R_0}\vec{e}+tq)-\frac{\int_Mh e^{U_{t,q} -G_t^{(2)} }\phi_t \mathrm{d}v_g}{\int_Mh e^{U_{t,q} -G_t^{(2)} }\mathrm{d}v_g}\Big)
\int_{B_{\Gamma_{t,q}}(0)}\Big(2\nabla\eta_2\cdot\nabla\overline{\chi_{t,q}}+\eta_2(\Delta\overline{\chi_{t,q}})\Big)\mathrm{d}z.
\end{aligned}
\end{equation}
From the definition of $\overline{\chi_{t,q}}$ and integration by parts, we see that
\begin{equation*}
\begin{aligned}
\int_{B_{\Gamma_{t,q}}(0)}\left(2\nabla\eta_2\cdot\nabla\overline{\chi_{t,q}}+\eta_2(\Delta\overline{\chi_{t,q}})\right)
=~&\int_{B_{\Gamma_{t,q}}(0)\setminus B_{\frac{\Gamma_{t,q}}{2}}(0)}
\left(2\nabla\eta_2\cdot\nabla\overline{\chi_{t,q}}+\eta_2(\Delta\overline{\chi_{t,q}})\right)\\
=~&\int_{B_{\Gamma_{t,q}}(0)\setminus B_{\frac{\Gamma_{t,q}}{2}}(0) }\nabla\eta_2\cdot\nabla\overline{\chi_{t,q}}=O(1),
\end{aligned}
\end{equation*}
where we used $|\nabla \eta_2(z)|\le O(1)\Big(\frac{1}{1+|z|}\Big)$ and $|\nabla\overline{\chi_{t,q}}|=O(t)$.
We also note that $|\eta_2(z)|\le O(1)( \ln(1+|z|))$. Combined with $|\nabla^2\overline{\chi_{t,q}}|=O(t^2)$,  we get that
\begin{equation}
\label{integration_by_parts2}
\begin{aligned}
&\int_{B_{\Gamma_{t,q}}(0)\setminus B_{\frac{\Gamma_{t,q}}{2}}(0) }|\eta_2(z)\Delta\overline{\chi_{t,q}}|+2|\nabla\eta_2\cdot\nabla\overline{\chi_{t,q}}| \mathrm{d}z=O(|\ln t|).
\end{aligned}
\end{equation}
From Lemma \ref{improved_int_limit_of_psi_t}, we have $\int_{B_{\Gamma_{t,q}}(0)} \Delta\psi_t(z) \mathrm{d}z=o(|\ln t|^{-1})$.
Together with Lemma \ref{diff_of_phi_t}-(i), we get for $x\in B_{t{R_0}} (tq)\setminus B_{\frac{t{R_0}}{2}}(tq)$,
\begin{equation}
\begin{aligned}
\label{small_o_log}
&|\phi_t(x)-\phi_t(tq+t{R_0}\vec{e})|\\
&\le O(1)\Bigg(\left|\left(\ln\frac{|t{R_0}\vec{e}|}{|x-tq|}\right)\int_{B_{\Gamma_{t,q}}(0)}\Delta\psi_t(z) \right|+t^{\frac{2(p-1)}{p}}|\ln t|+O(\lambda_{t,q}e^{-\frac{\alpha\lambda_{t,q}}{4}})\Bigg)\\
&=o(|\ln t|^{-1}).
\end{aligned}
\end{equation}
From \eqref{int_psi_with_kernal}-\eqref{small_o_log} and Lemma \ref{vanish_outside}, we obtain as $t\to0$,
\begin{equation*}
\begin{aligned}\int_{B_{\Gamma_{t,q}}(0)}\frac{16\psi_t(z)\overline{\chi_{t,q}}(z)(1-|z|^2)}{(1+|z|^2)^3}\mathrm{d}z=o(1).
\end{aligned}
\end{equation*}
By \eqref{bdd_of_psi_t}, we have $\|\psi_t\|_{L^{\infty} (B_{\Gamma_{t,q}}(0))}\le O(1)$. Since $\psi_t(z)\to d_0Y_0(z)$ in  $C^{0,\beta}_{\textrm{loc}}(\mathbb{R}^2 )$ from Lemma \ref{limitfunctionofpsi_t}, we can see that
\begin{equation*}
\begin{aligned}
0=\lim_{t\to0}\int_{B_{\Gamma_{t,q}}(0)}\frac{ \psi_t(z)\overline{\chi_{t,q}}(z)(1-|z|^2)}{(1+|z|^2)^3}\mathrm{d}z
=d_0\int_{\RN}\frac{Y_0(z)^2}{(1+|z|^2)^2}\mathrm{d}z,
\end{aligned}
\end{equation*}
which implies $d_0=0$ and $\psi_t(z)\to 0$ in  $C^{0,\beta}_{\textrm{loc}}(\mathbb{R}^2 )$. Using Lemma \ref{vanish_outside}, we also get that  $\overline{\phi_t}(z)= \psi_t(z)+\frac{\int_Mh e^{U_{t,q} -G_t^{(2)} }\phi_t \mathrm{d}v_g}{\int_Mh e^{U_{t,q} -G_t^{(2)} }\mathrm{d}v_g}\to 0$ in  $C^{0,\beta}_{\textrm{loc}}(\mathbb{R}^2 )$ Thus, Lemma \ref{vanish_insidew}-(i) holds.
\medskip

\noindent (ii). We can easily get Lemma \ref{vanish_insidew}-(ii) from $\overline{\phi_t}(z)=\phi_t(\Lambda_{t,q}^{-1}z+tq )\to0$ in  $C^{0,\beta}_{\textrm{loc}}(\mathbb{R}^2 )$.
\medskip

\noindent (iii). We shall prove $\lim_{t\to0}\|\phi_t\|_{L^{\infty}(M)}=0$ by contradiction. Suppose  that
$$\|\phi_t\|_{L^{\infty}(M)}= |\phi_t(x_t) |\ge c_0\ \ \textrm{for small}\ \ t>0,$$
where $c_0>0$ is a constant  independent of $t>0$. Then by Lemma \ref{vanish_outside}-(ii) and Lemma \ref{vanish_insidew}-(ii), we have
\begin{equation}
\label{location}
|x_t-tq|\le \frac{tR_0}{2}, \quad \lim_{t\to0}\frac{|x_t-tq|}{t^2}=+\infty.
\end{equation}
By Lemma \ref{diff_of_phi_t}-(ii),  we  have
\begin{equation}
\begin{aligned}
\label{inside_smallarea}
&|\phi_t(x_t)-\phi_t(tq)| \\&\le O(1)\Big(
\Big|\int_{B_{\Gamma_{t,q}}(0)} \ln|z| \Delta\psi_t(z)\mathrm{d}z\Big|+|\ln t|\Big|\int_{B_{\Gamma_{t,q}}(0)}\Delta\psi_t(z)\mathrm{d}z\Big|\Big)\\
&+O(1)\Big(t^{\frac{2(p-1)}{p}}|\ln t| +\Big(\frac{|x_t-tq|}{t^2}\Big)^{-\frac{\alpha}{2}}\ln \frac{|x_t-tq|}{t^2}\Big).
\end{aligned}
\end{equation}
From Lemma \ref{improved_int_limit_of_psi_t}, we have
\begin{equation}
\begin{aligned}
\label{inside_smallarea2}
\left|\int_{B_{\Gamma_{t,q}}(0)}\Delta\psi_t(z)\mathrm{d}z\right|=o(|\ln t|^{-1})
\end{aligned}
\end{equation}
Using equation \eqref{eq_for_psi} and  H\"{o}lder inequality, we get
\begin{equation*}
\begin{aligned}
\label{inside_smallarea3}
&\int_{B_{\Gamma_{t,q}}(0)}|\ln |z|||\Delta\psi_t(z)|\\
&\le O(1)\Big( \int_{B_{\Gamma_{t,q}}(0)}  \frac{O(\ln |  z| |\psi_t(z)|) }{(1+|z|^2)^2}\Big)
\\&+O(1)\Big(\int_{B_{\Gamma_{t,q}}(0)}|\ln|z||t^2e^{-\lambda_{t,q}} (|\overline{g_{t}}(z)|+\sum_{i=1}^2|c_{t,q,i}|| \overline{Z_{t,q,i}}(z)| )\Big)\\
&=O\left(\Big\|\frac{\psi_t(z)}{(1+|z|)^{1+\frac{\alpha}{2}}}\Big\|_{L^2(B_{\Gamma_{t,q}}(0))}+\|g_{t}\|_{Y_{\alpha,t,p}}+ \sum_{i=1}^2|c_{t,q,i}|\|Z_{t,q,i}\|_{Y_{\alpha,t,p}}\right).
\end{aligned}
\end{equation*}
Then we can apply Lemma \ref{vanish_insidew}-(i), Lemma \ref{magofcc}, and $\|g_{t}\|_{Y_{\alpha,t,p}}=o(|\ln  t|^{-1})$ to get
\begin{equation}
\begin{aligned}
\label{inside_smallarea4}
\int_{B_{\Gamma_{t,q}}(0)} |\ln |  z|||\Delta\psi_t(z)|\mathrm{d}z=o(1).
\end{aligned}
\end{equation}
Finally, using \eqref{location}-\eqref{inside_smallarea4} and Lemma \ref{vanish_insidew}-(ii), we obtain $\lim_{t\to0}|\phi_t(x_t)|=0$, which contracts to $|\phi_t(x_t) |\ge c_0$. Thus, we finish the whole proof.
\end{proof}

\medskip
\noindent
\textbf{Proof of Theorem \ref{thma}.} By  Lemma \ref{phitalphat} and  Lemma \ref{vanish_insidew}, we have
$$\lim_{t\to0}\left(\|\phi_t\|_{L^{\infty}(M)}+\|\phi_t\|_{X_{\alpha,t,p}}\right)=0,$$
which contradicts to the assumption \eqref{eqto1}. Therefore, we get the inequality \eqref{ineq} and it implies that $\mathbb{Q}_{t,q}\mathbb{L}_{t,q}$ is one-to-one from   $E_{\alpha,t,q,p}$ to $F_{\alpha,t,q,p}$.

To complete the proof of Theorem \ref{thma}, we follow the arguments in \cite{ly1}  to show that $\mathbb{Q}_{t,q}\mathbb{L}_{t,q}$ is onto from $E_{\alpha,t,q,p}$ to $ F_{\alpha,t,q,p}$. As in \cite{ly1}, we define $\widehat{\mathbb{L}}_{t,q}$ such that
$$\widehat{\mathbb{L}}_{t,q}(\phi):=\Delta\phi-\frac{8\Lambda_{t,q}^2\chi_{t,q}(x)}{(1+\Lambda_{t,q}^2|x-tq|^2)^2}\phi.$$
Then $\mathbb{Q}_{t,q}\widehat{\mathbb{L}}_{t,q}$ is an isomorphism from $E_{\alpha,t,q,p}$ to $F_{\alpha,t,q,p}$, and thus $\textrm{ind}(\mathbb{Q}_{t,q}\widehat{\mathbb{L}}_{t,q})=0.$ Moreover, we see that
\begin{equation*}
\mathbb{Q}_{t,q}\mathbb{L}_{t,q}\phi=\mathbb{Q}_{t,q}\widehat{\mathbb{L}}_{t,q}\phi+\mathbb{Q}_{t,q}(\mathbb{L}_{t,q}-\widehat{\mathbb{L}}_{t,q})\phi.
\end{equation*}
Since $\mathbb{Q}_{t,q}(\mathbb{L}_{t,q}-\widehat{\mathbb{L}}_{t,q})$ is a compact operator, we get
\[\textrm{dim}(\textrm{ker} (\mathbb{Q}_{t,q}\mathbb{L}_{t,q}))-\textrm{codim}(\textrm{ran}(\mathbb{Q}_{t,q}\mathbb{L}_{t,q}))
=\textrm{ind}(\mathbb{Q}_{t,q}\mathbb{L}_{t,q})
=\textrm{ind}(\mathbb{Q}_{t,q}\widehat{\mathbb{L}}_{t,q})=0.\]  By \eqref{ineq}, we have $\textrm{dim}(\textrm{ker} (\mathbb{Q}_{t,q}\mathbb{L}_{t,q}))=0$ and   $\textrm{codim}(\textrm{ran}(\mathbb{Q}_{t,q}\mathbb{L}_{t,q}))=0.$
As a consequence, we get $\mathbb{Q}_{t,q}\mathbb{L}_{t,q}$ is onto from $E_{\alpha,t,q,p}$ to $F_{\alpha,t,q,p}$ and it proves Theorem \ref{thma}.
$\hfill\square $

\section{Proof of Theorem \ref{theorem1.4}}\label{sec_proof}
In this section, we are going to  construct  a solution of \eqref{2.1} with the form
$$u_t =u_{t,q}^* -\int_Mu_{t,q}^*\mathrm{d}v_g+\phi_{t,q}=U_{t,q}+\phi_{t,q}.$$
Then we need to find  $\phi_{t,q}$ solving the following system:%
\begin{equation}
\label{33}
\mathbb{L}_{t,q}\left(\phi_{t,q}\right)
=g_{t,q}\left(\phi_{t,q}\right),
\end{equation}
where
\begin{equation*}
\begin{aligned}
g_{t,q}\left(\phi\right):=&-\Delta u_{t,q}^*+\rho-\frac{\rho he^{U_{t,q}-G_t^{(2)}+\phi}}{\int_Mh e^{U_{t,q} -G_t^{(2)} +\phi}\mathrm{d}v_g}\\
&+\frac{\rho he^{U_{t,q}-G_t^{(2)}}}{\int_Mh e^{U_{t,q} -G_t^{(2)} }\mathrm{d}v_g}\left(\phi-\frac{\int_Mh e^{U_{t,q} -G_t^{(2)} }\phi \mathrm{d}v_g}{\int_Mh e^{U_{t,q} -G_t^{(2)} }\mathrm{d}v_g}\right).
\end{aligned}
\end{equation*}
\medskip

\noindent Now we are going to prove Theorem \ref{theorem1.4} with the following steps.

\noindent\textit{Step 1}. Let $0<\alpha<\min\{\frac{1}{2},\frac{4(p-1)}{p}\}$.  We claim that there exists $t_1\in(0,t_0)$ such that if $0<t<t_1$ and  $q\in B_{t|\ln t|}(0)$, then  there is $%
\phi_{t,q} \in E_{\alpha,t,q,p}$ satisfying
\begin{equation}
\left\{\begin{array}{l}
\label{eq_for_existence_of_eta}
\mathbb{Q}_{t,q}\mathbb{L}_{t,q}(\phi_{t,q})=\mathbb{Q}_{t,q}(g_{t,q}(\phi_{t,q})),\\
\Vert\phi_{t,q}\Vert_{L^{\infty}(M)}+\|\phi_{t,q}\|_{X_{\alpha,t,p}}\le   t^{\frac{2}{p}}|\ln t|^2.
\end{array}\right.
\end{equation}
Using Theorem \ref{thma}, we can write \eqref{eq_for_existence_of_eta} as
\begin{equation*}
\begin{aligned}
\phi_{t,q}&=B_{t,q}(\phi_{t,q}), \ \ \textrm{where} \ \ B_{t,q}(\phi_{t,q}):=(\mathbb{Q}_{t,q}\mathbb{L}_{t,q})^{-1}(\mathbb{Q}_{t,q}g_{t,q}(\phi_{t,q})).
\end{aligned}
\end{equation*}
From Theorem \ref{thma} and Lemma \ref{pronorm}, we have
\begin{equation}
\begin{aligned}
\label{bnorm}
\|B_{t,q}(\phi_{t,q})\|_{L^{\infty}(M)}+ \Big\|B_{t,q}(\phi_{t,q})\|_{X_{\alpha,t,p}} \leq
C|\ln t|\|g_{t,q}(\phi_{t,q})\|_{Y_{\alpha,t,p}}
\end{aligned}
\end{equation}
for some generic constant $C$. Let
\begin{equation}
\begin{aligned}
&S_{{t}}:=\Big\{ \phi  \in  E_{\alpha,t,q,p}\ \Big|\Vert \phi\Vert_{L^{\infty}(M)}+\|\phi\|_{X_{\alpha,t,p}}\le  t^{\frac{2}{p}}|\ln t|^2\Big\}.
\end{aligned}
\end{equation}
To prove the  existence of the solution  for  \eqref{eq_for_existence_of_eta}, we will show that $B_{t,q}$ is a contraction map from $S_{{t}}$ to $S_{{t}}$ provided $t>0$ is sufficiently small and  $q\in B_{t|\ln t|}(0)$. By Lemma \ref{basic_est_for_app}, we have for $x\in M\setminus B_{t{R_0}}(tq)$,
\begin{equation}\label{g1}
\begin{aligned}
&g_{t,q}(\phi )(x)\\
&=8\pi\theta_{t,q}+(\rho-8\pi)\frac{he^{w}}{\int_Mhe^{w}\mathrm{d}v_g}-\rho\frac{h(x)e^{U_{t,q}(x)-G_t^{(2)}(x)}}{\int_Mh e^{U_{t,q} -G_t^{(2)} }\mathrm{d}v_g }+O(1)(\|\phi\|^2_{L^{\infty}(M)})\\
&= O(1)\left(1_{B_{\mathbf{r}_0}(0)}(x)\Big(\frac{t|q|}{|x-tq|}+ \frac{t^2}{|x-tq|^2}\Big)+t^2|\ln t|+t|q|+\|\phi\|^2_{L^{\infty}(M)}\right).
\end{aligned}
\end{equation}
Using Lemma \ref{basic_est_for_app} again, we see that  for $x\in B_{t{R_0}}(tq)$,
\begin{equation*}
\begin{aligned}
g_{t,q}(\phi)(x)
&=8\pi\theta_{t,q}+(\rho-8\pi)\frac{he^{w}}{\int_Mhe^{w}\mathrm{d}v_g}+\frac{8\Lambda_{t,q}^2}{(1+\Lambda_{t,q}^2|x-tq|^2)^2}\\
&\quad-\frac{e^{\lambda_{t,q}}t^{-2}\rho H_{t,q}\left(\frac{x}{t}\right)(1+O(\|\phi\|^2_{L^{\infty}(M)}+t^2))}{(1+\Lambda_{t,q}^2|x-tq|^2)^2(1+ \mathfrak{A}_{t,q})},
\end{aligned}
\end{equation*}
which implies for $z\in B_{\Gamma_{t,q}}(0)$,
\begin{equation}
\label{g2}
\begin{aligned}
&t^2e^{-\lambda_{t,q}}\overline{ g_{t,q}(\phi )}(z)\\
&=t^2e^{-\lambda_{t,q}} g_{t,q}(\phi )(\Lambda_{t,q}^{-1}z+tq )\\&=
\frac{ \rho H_{t,q}(q)}  {(1+|z|^2)^2} -\frac{\rho H_{t,q}(\Lambda_{t,q}^{-1}t^{-1}z+q) (1+O(1)(\|\phi\|^2_{L^{\infty}(M)}+t^2))}{(1+|z|^2)^2(1+ \mathfrak{A}_{t,q})}+O(t^4)\\
&=\frac{ \rho H_{t,q}(q)}  {(1+|z|^2)^2}\Big(1-\frac{1}{(1+\mathfrak{A}_{t,q})}\Big)
-\frac{\rho \nabla H_{t,q}(q)\cdot(\Lambda_{t,q}^{-1}t^{-1}z)}{(1+|z|^2)^2(1+ \mathfrak{A}_{t,q})}\\
&\quad+\frac{O(e^{- \lambda_{t,q}}|z|^2+ \|\phi\|^2_{L^{\infty}(M)}+t^2)}{(1+|z|^2)^2} +O(t^4),
\end{aligned}
\end{equation}
here we used the Taylor expansion of $ H_{t,q}(C_{t,q}e^{-\frac{\lambda_{t,q}}{2}}z+q)$.

Recall $\mathfrak{A}_{t,q}=O(1)(t|q|+t^2|\ln t|)$. We also note that $\nabla H_{t,q}(q)=O(1)(t+|q|)$. Then $|q|\le t|\ln t|$ and  \eqref{g1}-\eqref{g2} yield
\[\|g_{t,q}(\phi )\|_{Y_{\alpha,t,p}}=O(1)(t^{\frac{2}{p}}+\|\phi\|^2_{L^\infty(M)}).\]
Combined with \eqref{bnorm}, we see that
$B_{t,q}: S_{{t}}\longrightarrow S_{{t}}$ provided $t$ is sufficiently small. Following a similar argument, we also get that if $\phi_1$, $\phi_2\in S_t$, then
\begin{equation*}
\|g_{t,q}(\phi_1 )-g_{t,q}(\phi_2)\|_{Y_{\alpha,t,p}}=
O(1)\left( (\|\phi_1\|_{L^\infty(M)}+\|\phi_2\|_{L^\infty(M)})\|\phi_1-\phi_2\|_{L^\infty(M)}\right).
\end{equation*}
Together with \eqref{bnorm}, we see that if $t>0$ is small, then
\begin{equation*}
\begin{aligned}
&\|B_{t,q}(\phi_1 )-B_{t,q}(\phi_2 )\|_{L^{\infty}(M)}+\|B_{t,q}(\phi_1 )-B_{t,q}(\phi_2 )\|_{X_{\alpha,t,p}}\\&\le
\frac{1}{2}\Big(\| \phi_1 -\phi_2 \|_{L^{\infty}(M)}+\| \phi_1 - \phi_2  \|_{X_{\alpha,t,p}}
\Big)\ \ \textrm{for any}\ \ \phi_1, \phi_2\in S_t.
\end{aligned}
\end{equation*}
Therefore, we have $B_{t,q}$ is a contraction map from $S_t$ to $S_t$ for small $t>0$. As a consequence, we get the existence of $\phi_{t,q} \in E_{\alpha,t,q,p}$ satisfying \eqref{eq_for_existence_of_eta}.
\medskip

\noindent\textit{Step 2}. In \textit{Step 1}, we have proved that for some $t_1\in(0,t_0)$ such that if $t\in(0,t_1)$ and $q\in B_{t|\ln t|}(0)$, then there exist
$$(\phi_{t,q},c_{t,q,1},c_{t,q,2})\in E_{\alpha,t,q,p}\times\mathbb{R}\times\mathbb{R}$$
satisfying \eqref{c} and \eqref{intcfind} with $g_{t}(x)$ is replaced by $g_{t,q}(\phi_{t,q})$. For convenience, we write $g_{t,q}(\phi_{t,q})$ by $g_{t,q}$. To complete the proof of Theorem \ref{theorem1.4}, it is enough to find $q\in\mathbb{R}^2$
satisfying $c_{t,q,i}\equiv0$ for $i=1,2$. We denote
\[\psi_{t,q}(z)=\overline{\phi}_{t,q}(z)-\frac{\int_Mh e^{U_{t,q} -G_t^{(2)} }\phi_{t,q} \mathrm{d}v_g}{\int_Mh e^{U_{t,q} -G_t^{(2)} }\mathrm{d}v_g}.\]
Then  $\psi_{t,q}$ satisfies \eqref{eq_for_psi}  with $\overline{g_{t}}(z)$ replaced by $\overline{g_{t,q}}(z)$. As in the proof of Lemma \ref{magofcc}, we apply Lemma \ref{ztqi} and the integration by parts to get
\begin{equation}
\label{f0}
\begin{aligned}
&c_{t,q,j}\int_{B_{\Gamma_{t,q}}(0)} |\nabla \overline{\chi_{t,q}}(z)Y_j(z)|^2+\frac{8(\overline{\chi_{t,q}}(z))^3(Y_j(z))^2}{(1+|z|^2)^2}\mathrm{d}z
\\&=\int_{B_{\Gamma_{t,q}}(0)}\psi_{t,q}(z)
\Big\{\Delta\overline{\chi_{t,q}}(z)\frac{z_j}{(1+|z|^2)}+2\nabla\overline{\chi_{t,q}}(z)\cdot\nabla \Big(\frac{z_j}{(1+|z|^2)}\Big) \Big\}\mathrm{d}z
\\&\quad+\int_{B_{\Gamma_{t,q}}(0)}\psi_{t,q}(z)\Big\{\overline{\chi_{t,q}}(z)\Delta\Big(\frac{z_j}{(1+|z|^2)}\Big) +\frac{8\overline{\chi_{t,q}}(z)z_j}{(1+|z|^2)^3} \Big\}\mathrm{d}z
\\&\quad+\int_{B_{\Gamma_{t,q}}(0)}\frac{O(1)(|\psi_{t,q}(z)Y_j(z)|  (e^{-\frac{\lambda_{t,q}}{2}}|z|+t|q|+t^2|\ln t|))}{(1+|z|^2)^2}\mathrm{d}z
\\&\quad-\int_{B_{\Gamma_{t,q}}(0)}\Lambda_{t,q}^{-2}\overline{g_{t,q}}(z) \overline{\chi_{t,q}}(z)\frac{z_j}{(1+|z|^2)}\mathrm{d}z
\\&=\mathfrak{N}_{t,q,1}+\mathfrak{N}_{t,q,2}+\mathfrak{N}_{t,q,3}+\mathfrak{N}_{t,q,4}.
\end{aligned}
\end{equation}
Next, we shall compute the right hand side of \eqref{f0} term by term.
\medskip

\noindent (i) By $\|\psi_{t,q}\|_{L^\infty(B_{\Gamma_{t,q}}(0))}=O(1)( \|\phi_{t,q}\|_{L^{\infty}(M)}),$ we have $\mathfrak{N}_{t,q,3}=O(1)(t\|\phi_{t,q}\|_{L^{\infty}(M)}).$

\noindent (ii) Using $|\nabla \overline{\chi_{t,q}}|=O(t)$ and $|\nabla^2 \overline{\chi_{t,q}}|=O(t^2)$, we get $\mathfrak{N}_{t,q,1}=O(1)(t\|\phi_{t,q}\|_{L^{\infty}(M)}).$

\noindent (iii) We can easily get $\mathfrak{N}_{t,q,2}=0$ from $\Delta\Big(\frac{z_j}{(1+|z|^2)}\Big) +\frac{8z_j}{(1+|z|^2)^3}=0$.

\noindent (iv) By using \eqref{g2} and $\overline{\chi_{t,q}}(z)=\overline{\chi_{t,q}}(|z|)$, we have
\begin{equation*}
\begin{aligned}
\mathfrak{N}_{t,q,4}=-&\int_{B_{\Gamma_{t,q}}(0)}\Lambda_{t,q}^{-2}\overline{ g_{t,q} }(z)\overline{\chi_{t,q}}(z)\frac{z_j}{(1+|z|^2)}\mathrm{d}z\\
=&~-C_{t,q}^2\int_{B_{\Gamma_{t,q}}(0)}\Big\{\frac{ \rho H_{t,q}(q)}{(1+|z|^2)^2}\Big(1-\frac{1}{(1+ \mathfrak{A}_{t,q})}\Big)
-\frac{\rho \nabla H_{t,q}(q)\cdot(t^{-1}\Lambda_{t,q}^{-1}z)}{(1+|z|^2)^2(1+ \mathfrak{A}_{t,q})}\\
&+\frac{O(e^{-\lambda_{t,q}}|z|^2+\|\phi_{t,q}\|^2_{L^{\infty}(M)}+t^2)}{(1+|z|^2)^2}+O(t^4)\Big\}\overline{\chi_{t,q}}(|z|)\frac{z_j}{(1+|z|^2)}\mathrm{d}z\\
=&C_{t,q}^2\int_{B_{\Gamma_{t,q}}(0)}\frac{\rho \nabla_j H_{t,q}(q)C_{t,q}e^{-\frac{\lambda_{t,q}}{2}}z_j^2 \overline{\chi_{t,q}}(|z|)   }{(1+|z|^2)^3(1+ \mathfrak{A}_{t,q})}\mathrm{d}z+O(\|\phi_{t,q}\|^2_{L^{\infty}(M)}+t^2).
\end{aligned}
\end{equation*}
Fix $p\in(1,2)$.  From \eqref{f0}, (i)-(iv), and $\|\phi_{t,q}\|_{L^{\infty}(M)}\le t^{\frac{2}{p}}|\ln t|^2$, we get
\begin{equation*}
\begin{aligned}
&c_{t,q,j}\int_{B_{\Gamma_{t,q}}(0)} |\nabla \overline{\chi_{t,q}}(z)Y_j(z)|^2+\frac{8(\overline{\chi_{t,q}}(z))^3(Y_j(z))^2}{(1+|z|^2)^2}\mathrm{d}z\\
&= \frac{\nabla_j H_{t,q}(q)}{H_{t,q}(q)}e^{-\frac{\lambda_{t,q}}{2}}\int_{B_{\Gamma_{t,q}}(0)} \frac{8C_{t,q}z_j^2\overline{\chi_{t,q}}(z) }{(1+|z|^2)^3(1+ \mathfrak{A}_{t,q})}  \mathrm{d}z\\
&\quad+O(1)(t\|\phi_{t,q}\|_{L^{\infty}(M)}+\|\phi_{t,q}\|^2_{L^{\infty}(M)}+t^2)\\
&=\frac{\nabla_j H_{t,q}(q)}{H_{t,q}(q)}e^{-\frac{\lambda_{t,q}}{2}}\int_{B_{\Gamma_{t,q}}(0)} \frac{8   C_{t,q}z_j^2\overline{\chi_{t,q}}(z) }{(1+|z|^2)^3(1+ \mathfrak{A}_{t,q})} \mathrm{d}z+\mathfrak{O}_{t,q,j},
\end{aligned}
\end{equation*}where $\mathfrak{O}_{t,q,j}=O(t^2)$.
We note that $\nabla H_{t,q}(0)=O(t)$ and  $\nabla^2 H_{t,q}(0)$ is invertible for small $t>0$ and $q\in B_{t|\ln t|}(0)$. Therefore, for small $t>0$, there is $q_t=O(t)$ satisfying \[\frac{\nabla_j H_{t,q_t}(q_t)}{H_{t,q_t}(q_t)}e^{-\frac{\lambda_{t,q_t}}{2}}\int_{B_{\Gamma_{t,q_t}}(0)} \frac{8   C_{t,q_t}z_j^2\overline{\chi_{t,q_t}}(z) }{(1+|z|^2)^3(1+ \mathfrak{A}_{t,q_t})} \mathrm{d}z+\mathfrak{O}_{t,q_t,j}=0 \ \textrm{for}\ \ j=1,2,\] and thus $c_{t,q_t,j}=0$ for $j=1,2$. Now we complete the proof of Theorem \ref{theorem1.4}.\hfill$\square$

\end{document}